\documentclass{article}
\usepackage[utf8]{inputenc}
\usepackage[margin=1in]{geometry} 
\usepackage{amsmath,amsthm,amssymb, amscd, url,mathrsfs, mathtools, tikz-cd, float, hyperref, blkarray} 
\usepackage[alphabetic]{amsrefs}
\usepackage{color}

\usetikzlibrary{matrix}

\newcommand{\Z}{\mathbb{Z}}
\newcommand{\mf}[1]{\mathfrak{#1}}

\DeclareMathOperator{\supp}{supp}
\DeclareMathOperator{\Tor}{Tor}

\DeclareMathOperator{\set}{set}
\DeclareMathOperator{\lex}{lex}
\DeclareMathOperator{\sgn}{sgn}
\DeclareMathOperator{\rk}{rk}

\newtheorem{theorem}{Theorem}[section]
\newtheorem{corollary}[theorem]{Corollary}
\newtheorem{lemma}[theorem]{Lemma}
\newtheorem{proposition}[theorem]{Proposition}

\theoremstyle{definition}

\newtheorem{example}[theorem]{Example}

\newtheorem*{theorem*}{Theorem}

\title{Monomial Cycles in Koszul Homology}
\author{Jacob Zoromski}

\date{}

\begin{document}

\maketitle

\abstract{In this paper we study monomial cycles in Koszul homology over a monomial ring. The main result is that a monomial cycle is a boundary precisely when the monomial representing that cycle is contained in an ideal we introduce called the boundary ideal. As a consequence, we obtain necessary ideal-theoretic conditions for a monomial ideal to be Golod. We classify Golod monomial ideals in four variables in terms of these conditions. We further apply these conditions to symmetric monomial ideals, allowing us to classify Golod ideals generated by the permutations of one monomial. Lastly, we show that a class of ideals with linear quotients admit a basis for Koszul homology consisting of monomial cycles. This class includes the famous case of stable monomial ideals as well as new cases, such as symmetric shifted ideals.   }

\section{Introduction}

Let $k$ be a field, $S = k[x_1,\dots,x_n]$, and $R = S/I$ for a monomial ideal $I \subset S$ be generated in degree at least two.  In this paper, we study the Koszul homology of $R$ by focusing on monomial cycles. We are motivated by the following questions:

\begin{enumerate}
\item When does a monomial cycle represent a (non-) zero class in Koszul homology?

\item

\begin{enumerate}
\item
When do products of monomial cycles vanish in Koszul homology? 

\item
When does the vanishing of such products imply $R$ is Golod?

\end{enumerate}

\item Which ideals admit a basis for Koszul homology consisting of monomial cycles?

\end{enumerate}

Previous study of monomial cycles in Koszul homology focused on  Question 3. While we contribute to that question in this paper, we will primarily focus on the first two questions. Our main result answers Question 1 in terms of an inclusion in an ideal we introduce called the boundary ideal. 

\begin{theorem}
\label{main thm}
Let $u$ be a monomial of $S$. A cycle $\bar u e_{i_1} \wedge \dots \wedge e_{i_p}$ is a boundary in the Koszul complex of $R$ if and only if $u$ is contained in the boundary ideal $B^{n,\{i_1,\dots,i_p\}}_I$.
\end{theorem}

The boundary ideal for $p = n$ or $p = n-1$ is easy to compute. There are no boundaries in homological degree $n$, so $B^{n,\{1,\dots,n\}}_I = I$. In three variables, a monomial cycle $\bar u e_1 \wedge e_2$ can be written as the boundary of a single element $ \bar v e_1 \wedge e_2 \wedge e_3$. This means that
\[
\partial( \bar v e_1 \wedge e_2 \wedge e_3 \wedge e_4) = \overline{x_1 v} e_2 \wedge e_3 - \overline{x_2 v} e_1 \wedge e_1 + \overline{x_3 v} e_1 \wedge e_2 = \bar u e_1 \wedge e_2
\]
We must have $u = x_3 v$ and $x_1 v,x_2 v$ are in $I$, so  either $u$ is in $I$ or $u$ is in $x_3[I:(x_1,x_2)]$. Therefore, $B_I^{3,\{1,2\}} = I + x_3[I:(x_1,x_2)]$. 

\begin{example}
Let $I = (x_1x_3,x_2x_3) \subset k[x_1,x_2,x_3]$. Now $I:(x_1,x_2) = (x_3)$, so $\bar x_3 e_1 \wedge e_2$ is a monomial cycle. The boundary ideal associated to $\{1,2\}$ is  $I + x_3[I:(x_1,x_2)] = (x_1x_3,x_2x_3,x_3^2)$. So,  $\bar x_3 e_1 \wedge e_2$ is not a boundary, while  $\bar x_3^2 e_1 \wedge e_2$ is: it is the boundary of $\bar x_3 e_1 \wedge e_2 \wedge e_3$.

\end{example}

\par
In general, computing boundary ideals is difficult. Their structure is governed by the circuits of the full simplicial matroid $S^n_{n-p}[k]$. Simplicial matroids form an important and well-studied class of matroids; see \cite{cordovillindstrom} for a survey of results and \cite{reiner} for a more recent example. We apply results about full simplicial matroids from Crapo and Rota \cite{cr}, Cordovil \cite{cordovil}, Cordovil and Las Vergnas \cite{clv}, and Lindstr\"om \cite{lind} to obtain an explicit recipe for computing the boundary ideal when $p=2$ and $p=n-2$,  and to demonstrate that the boundary ideal depends on the characteristic of $k$ for $3 \leq p \leq n-3$. 

\begin{example}
The circuits of $S^4_2[k]$ are in bijection with the 2-color graphs of four vertices. To construct $B_I^{4,\{1,2\}}$, we take those circuits containing $\{1,2\}$, i.e. the 2-color graphs where 1 and 2 are different colors, as shown below.

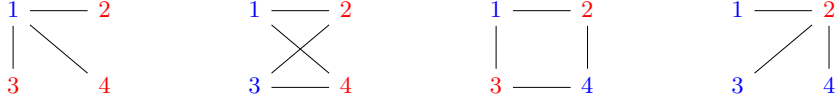
\begin{tikzcd}
{\color{blue}1} \arrow[r, no head] \arrow[rd, no head] \arrow[d, no head] & {\color{red}2} &  & {\color{blue}1} \arrow[r, no head] \arrow[rd, no head] & {\color{red}2} &  & {\color{blue}1} \arrow[r, no head] \arrow[d, no head] & {\color{red}2}                    &  & {\color{blue}1} \arrow[r, no head]  & {\color{red}2}                    \\
{\color{red}3}                                                           & {\color{red}4} &  & {\color{blue}3} \arrow[ru, no head] \arrow[r, no head] & {\color{red}4} &  & {\color{red}3} \arrow[r, no head]                    & {\color{blue}4} \arrow[u, no head] &  & {\color{blue}3} \arrow[ru, no head] & {\color{blue}4} \arrow[u, no head]
\end{tikzcd}

\end{example}

The boundary ideal is the intersection of certain ideals associated to each circuit:
\begin{align*}
B^{4,\{1,2\}}_I &= (I + x_3 [I:x_2] + x_4 [I:x_2]) \\ 
& \cap (I + x_4 [I: x_2] + x_3 [I:x_1] + x_3 x_4 [I: x_1 x_2])\\
& \cap (I + x_3 [I: x_2] + x_4 [I:x_1] + x_3 x_4 [I:x_1 x_2]) \\
& \cap (I + x_3 [I:x_1] + x_4 [I:x_1]) \\
&= I + (x_3,x_4)[I:(x_1,x_2)] + (x_3x_4)[I:x_1x_2] \cap ( x_3[I:x_1] \cap x_4[I:x_2] + x_3[I:x_2] \cap x_4[I:x_1] )
\end{align*}

\par
The differential graded algebra structure of the Koszul complex naturally descends to an algebra structure on Koszul homology. Since a product of monomial cycles is again a monomial cycle, Theorem \ref{main thm} supplies an answer to Question 2(a).

\begin{theorem}
\label{necessary cond}
Let $I$ be a monomial ideal. Then all products of monomial cycles in Koszul homology vanish if and only if the inclusions 
\[
[I:(x_1,\dots,x_q)][I:(x_{q+1},\dots,x_p)] \subset B_I^{n,\{1,2,\dots,p\}}
\]
hold for all $2 \leq p \leq n$ and for all $1 \leq q < p$ and for all permutations of the indices.
\end{theorem}

The vanishing of products of Koszul homology classes is closely related to the study of infinite free resolutions. Of particular interest in this regard is the Golod condition. When $I$ (or $R$) is Golod, all products of Koszul homology classes of homological degree at least one vanish. The implication cannot be reversed, however, as a famous example of Katth\"an \cite{katthan} shows. In the same paper, the author demonstrates that Golodness is equivalent to the vanishing of Koszul homology products in four variables. We show something stronger: to check Golodness in four variables, it is enough to check that products of monomial cycles vanish, positively answering Question 2(b) in this case. Therefore, as a consequence of Theorem \ref{necessary cond}, we obtain a classification of Golod monomial ideals in four variables.

\begin{theorem}
\label{four var}
Let $I$ be a monomial ideal in $k[x_1,x_2,x_3,x_4]$. Then $I$ is Golod if and only if the following hold under all permutations of the indices:
    \[
    [I:x_1][I:(x_2,x_3,x_4)] \subset I
    \]
    \[
    [I:(x_1,x_2)][I:(x_3,x_4)] \subset I
    \]
    \[
    [I:x_1][I:(x_2,x_3)] \subset I + x_4[I:(x_1,x_2,x_3)] 
    \]
    \[
    [I:x_1][I:x_2] \subset I + (x_3,x_4)[I:(x_1,x_2)] + (x_3x_4)[I:x_1x_2] \cap ( x_3[I:x_1] \cap x_4[I:x_2] + x_3[I:x_2] \cap x_4[I:x_1] )
    \]
\end{theorem}

Theorem \ref{four var} is an extension of the classification of three variable Golod monomial ideals by Dao and De Stefani \cite{dao-destefani-2020} in terms of ideal inclusions. A key element in their proof was showing that for all monomial ideals in three variables the Koszul homology admits a basis of monomial cycles (answering Question 3 positively in this case). What's surprising is that Theorem \ref{four var} holds despite the fact that not all Koszul homology classes can be represented by a monomial cycle in four or more variables, as the following example shows.

\begin{example}
Let $I = (x_1 x_3, x_1 x_4, x_2 x_3, x_2 x_4)$. As pointed out in \cite{dao-destefani-2020}, the non-zero homology class $[\bar x_1 e_2 \wedge e_3 \wedge e_4 - \bar x_2 e_1 \wedge e_3 \wedge e_4]$ cannot be expressed as a monomial cycle. We return to this example and give an alternate proof of this fact in terms of boundary ideals in Section \ref{section prelim}.

\end{example}

Another case where we can positively answer Question 2(b) is for certain (conjecturally, all) symmetric monomial ideals. A symmetric monomial ideal is a monomial ideal that is stable under the action of the symmetric group $\mf S_n$ which permutes the variables.  They can be described in terms of partitions $\lambda = (\lambda_1,\dots,\lambda_n) \in \Z_{\geq 0}^n$ with $\lambda_1 \geq \lambda_2 \geq \dots \geq \lambda_n$. That is, each symmetric monomial ideal $I$ has a unique set of partitions $\Lambda(I)$ such that 
\[
I = \sum_{\lambda \in \Lambda(I)} (x_1^{\lambda_{\pi(1)}} x_2^{\lambda_{\pi(2)}} \dots x_n^{\lambda_{\pi(n)}}: \pi \in \mf S_n)
\]
In the symmetric case, we are able to translate Theorem \ref{necessary cond} to conditions on the generating partitions $\Lambda(I)$. To state the theorem, we say $\ell(\lambda) := \max\{i : \lambda_i \neq 0\}$ is the length of a partition $\lambda$.

\begin{theorem}
\label{partition conditions}
Let $I$ be a symmetric monomial ideal. Then all products of monomial cycles in Koszul homology vanish if and only if for each $2 \leq p \leq n$, one of the following holds:

\begin{enumerate}
\item For all partitions $\lambda \in \Lambda(I)$ of length at most $\lfloor p/2 \rfloor$ with minimal first term among partitions in $\Lambda(I)$ of length at most $\lceil p/2 \rceil$, $\lambda_1 = \lambda_2$.

\item There exists a $\mu \in \Lambda(I)$ with $\lceil p/2 \rceil < \ell(\mu) \leq p$ such that $\mu_1$ smaller than the minimal first term among partitions of length at most $\lceil p/2 \rceil$ in $\Lambda(I)$.

\end{enumerate}

\end{theorem}

\begin{example}
\label{nonvanishing prod}
Let $\Lambda(I) = \{(3,0,0,0),(2,1,0,0)\}$, that is,
\[
I = (x_1^3,x_2^3,x_3^3,x_4^3, x_1^2x_2, x_1^2x_3, x_1^2x_4, x_1x_2^2, x_2^2x_3, x_2^2x_4, x_1x_3^2, x_2x_3^2, x_3^2x_4, x_1x_4^2, x_2x_4^2, x_3x_4^2)
\]
For $p=4$, condition 1 fails because the first two terms of the partition $(2,1,0,0)$ are not equal. Condition 2 fails as well, as there are no partitions of length greater than 2 in $\Lambda(I)$. So, there should be a non-vanishing product of monomial cycles. 
\par
Indeed, $[(\overline{x_1x_2} e_1 \wedge e_2][(\overline{x_3x_4} e_3 \wedge e_4)] = [\overline{x_1 x_2 x_3 x_4} e_1 \wedge e_2 \wedge e_3 \wedge e_4]$ is a non-zero homology class. In fact, this ideal was used by De Stafani in \cite{DESTEFANI20162289} as the first known example of a products of ideals that was not Golod, as $I = (x_1,x_2,x_3,x_4)(x_1^2,x_2^2,x_3^2,x_4^2)$.

\end{example}

We are unaware of any symmetric monomial ideals satisfying the conditions of Theorem \ref{partition conditions} but are not Golod. Thus, we conjecture that Question 2(b) holds for all symmetric monomial ideals. As evidence, we prove that symmetric monomial ideals generated by the $\mf S_n$-orbit of one monomial which satisfy the conditions of Theorem \ref{partition conditions} are Golod.

\begin{theorem}
\label{symm principal}
Let $I$ be a symmetric monomial ideal with $\Lambda(I) = \{\lambda\}$. Then $I$ is Golod if and only if $\ell(\lambda) > \lfloor n/2 \rfloor$ or both $\ell(\lambda) \leq \lfloor n/2 \rfloor$ and $\lambda_1 = \lambda_2$.

\end{theorem}

The resolutions of symmetric monomial ideals were studied extensively in recent years (for example, see \cite{Murai2020AnEH}, \cite{BIERMANN2020312}, \cite{MuraiSatoshi2020Btom}, \cite{RaicuClaudiu2021RoSm}, \cite{GalettoFederico2020Otig}, and \cite{KumarAshok2013MBno}). Our results are, as far as we are aware, the first investigation into resolutions over quotients of such ideals. See \cite{harada2023minimal} for a recent paper showing that generic symmetric principal ideals are Golod.
\par

In the final section, which can be read independently of the others, we return to Question 3. It was answered positively in \cite{aramova} and \cite{peeva} for stable monomial ideals, \cite{p-Borel} and \cite{POPESCU2008132} for certain $p$-Borel fixed monomial ideals, and \cite{dao-destefani-2020} for monomial ideals in three variables or fewer. As a generalization of the first of these,  we demonstrate that ideals with linear quotients which satisfy other technical conditions admit a monomial cycle basis in Koszul homology. This class includes, but is not limited to, the following well-studied ideals. 
\begin{theorem}
Stable monomial ideals, matroidal ideals, and symmetric shifted ideals  admit a monomial cycle basis in Koszul homology.
\end{theorem}
\par
Matroidal ideals were known to have linear quotients since the introduction of the property in \cite{HerzogTakayama}, and symmetric shifted ideals are a class of symmetric monomial ideals which were introduced in \cite{BIERMANN2020312}, where it was proved that they have linear quotients. Ideals with linear quotients are always Golod, so they necessarily satisfy the conditions of Theorem \ref{necessary cond}. Furthermore, the symmetric shifted ideals satisfy the conditions of Theorem \ref{partition conditions}, which is not hard to show directly.  It would be very interesting to answer Question 3 for ideals which are not necessarily Golod to see if Theorem \ref{necessary cond} would shed any light on whether they are Golod.

\section{Preliminaries}
\label{section prelim}
In this section, we provide background on Koszul homology, particularly its multigraded and multiplicative (differential graded algebra) structures. We will also give an extended example which will be used to illustrate both these structures and the main themes of the paper.

\subsection{The Koszul Complex}
We start with the Koszul complex. Let $k$ be a field and $S = k[x_1,\dots,x_n]$ the polynomial ring with the standard $\Z_{\geq 0}^n$ multigrading. Let $I \subset (x_1,\dots,x_n)^2 \subset S$ be a monomial ideal and $R = S/I$. Let $R^n$ be a free module with basis $e_1,\dots,e_n$ and $K^R$ be the Koszul complex with respect to $x_1,\dots,x_n$:
\[
K^R:  R \xleftarrow{\partial} \bigwedge^1 R^n \xleftarrow{\partial} \bigwedge^2 R^n \xleftarrow{\partial} \dots \xleftarrow{\partial} \bigwedge^n R^n \leftarrow 0
\]
For a subset $\sigma = \{\sigma_1,\dots,\sigma_p\}$ of $[n] = \{1,\dots,n\}$, we will use $e_{\sigma}$ or $e_{\sigma_1 \dots \sigma_p}$ to denote $e_{\sigma_1} \wedge \dots \wedge e_{\sigma_p}$. With this notation, the boundary map of $K^R$ is given by
\[
\partial(\bar f e_{\sigma}) = \sum_{j = 1}^p \sgn(\sigma,\sigma_j) \overline{x_{\sigma_j} f} e_{\sigma \backslash \sigma_j}
\]
where $\bar g$ is the image in $R$ of an element $g$ in $S$, and $\sgn(\sigma,\sigma_j)$ is $-1$ to the power of $\#\{s \in \sigma : s < \sigma_j\}$. The {\bf Koszul homology} of $R$ in degree $p$ is the $k$-vector space $H_p(K^R)$. 
\par
\subsection{Multigrading and Monomial Cycles}
We give $K^R$ a multigraded structure by letting $e_i$ have multidegree $\mathbf e_i$, the $i$th basis element of $\Z^n_{\geq 0}$. That is, an element has multidegree $\bf a$ when it is a linear combination of elements of the form $\bar f e_{\sigma_1 \dots \sigma_p}$ where $f \in S$ is the monomial of multidegree $\mathbf a - \sum_{j=1}^p \mathbf e_{\sigma_j}$, and furthermore $\partial$ preserves this multidegree. Since $I$ is a monomial ideal, its generators are multigraded, so $H_*(K^R)$ inherits the multigrading. To illustrate this, we will compute a multigraded basis in Koszul homology for an ideal which we will use as a recurring example in this section.
\par 
\begin{example}
\label{eg}
Let $J = (x_1 x_3, x_1 x_4, x_2 x_3, x_2 x_4) \subset k[x_1,x_2,x_3,x_4]$. The multigraded minimal free $S$-resolution of $R = S/J$ is of the form

\begin{tikzcd}[row sep=0mm]
 & S(-1,0,-1,0) & S(-1,-1,-1,0) &  \\
 S & \arrow[l] \oplus S(-1,0,0,-1) & \arrow[l] \oplus S(-1,-1,0,-1) & \arrow[l] S(-1,-1,-1,-1) & \arrow[l] 0 \\
 & \oplus S(0,-1,-1,0) & \oplus S(-1,0,-1,-1) & \\
 & \oplus S(0,-1,0,-1) & \oplus S(0,-1,-1,-1) & 
\end{tikzcd}

We can choose a multigraded basis for Koszul homology in each homological degree:

\par
$H_0(K^{S/J}) = \text{span}_k \langle[1]\rangle$
\par
$H_1(K^{S/J}) = \text{span}_k \langle [\bar x_3 e_1],[\bar x_4 e_1],[\bar x_3 e_2], [\bar x_4 e_2] \rangle$
\par 
$H_2(K^{S/J}) = \text{span}_k \langle [\bar x_3 e_{12}],[\bar x_4 e_{12}],[\bar x_1 e_{34}],[\bar x_2 e_{34}] \rangle$
\par 
$H_3(K^{S/J}) = \text{span}_k \langle [\bar x_1 e_{234} - \bar x_2 e_{134}] \rangle$ 

To show that this is in fact a basis, we will use Theorem \ref{main thm}

\end{example}

\par
In general, a monomial cycle $\bar u e_\sigma$ must satisfy $u \in I:(x_{i} : i \in \sigma)$. For it to be non-zero, we need that $u \not \in I$. Theorem \ref{main thm} tells us that for it to represent a non-zero class in Koszul homology, $u$ must avoid the larger ideal $B_I^{n,\sigma}$. As an example in four variables, $[\bar u e_{1}]$ is non-zero if and only if $u \not\in I + (x_2,x_3,x_4)[I:x_1]$. For the ideal $J$ above, $J:x_1 = (x_3,x_4)$ and $(x_2,x_3,x_4)[J:x_1] = (x_2 x_3,x_2 x_4, x_3 x_4, x_3^2, x_4^2)$. Theorem \ref{main thm} thus implies that $[\bar x_3 e_1],[\bar x_4 e_1]$ are non-zero homology classes. Repeating the argument but permuting the indices 1 and 2 gives us that $[\bar x_3 e_2],[\bar x_4 e_2]$ are non-zero homology classes. These classes are all linearly independent, so we verify the proposed multigraded basis for Koszul homology. 
\par 
The boundary ideals can get more complicated. In four variables, if $u \in I:(x_1,x_2)$ is monomial, $[\bar u e_{12}]$ is a non-zero homology class if and only if
\[
u \not \in I + (x_3,x_4)[I:(x_1,x_2)] + (x_3x_4)[I:x_1x_2] \cap ( x_3[I:x_1] \cap x_4[I:x_2] + x_3[I:x_2] \cap x_4[I:x_1] )
\]
Now $J:(x_1,x_2) = (x_3,x_4)$ and $J:(x_1x_2) = I$, so this condition simplifies to  $u \not\in J + (x_3,x_4)^2$. $u = x_3$ and $u = x_4$ are the only monomials satisfying these conditions, so $[\bar x_3 e_{12}],[\bar x_4 e_{12}]$ are non-zero homology classes. Repeating the argument while swapping the roles of 1 and 2 with 3 and 4 shows that $[\bar x_1 e_{34}]$ and $[\bar x_2 e_{34}]$ are non-zero homology classes. In homological degree three, $J:(x_1,x_2,x_3) = J$, so there are no non-zero monomial cycles of the form $\bar u e_{123}$; permuting the variables verifies that there are no non-zero monomial cycles in homological degree three. Thus, $H_3(K^{S/J})$ does not admit a basis of monomial cycles. 
\par
$H_1(K^{S/J})$ and $H_2(K^{S/J})$ do admit a basis consisting entirely of monomial cycles. This is a general phenomenon. For any monomial ideal $I$, the generators of $I$ correspond to basis elements of $H_1(K^R)$ by sending  a minimal monomial generator $u$ of $I$ to $[\frac{u}{x_a} e_a]$ for any choice of $a$ with $x_a$ dividing $u$. So, $H_1(K^R)$ always admits a basis consisting of monomial cycles. If $H_n(K^R) \neq 0$, then a multigraded basis element $[\bar f e_{1\dots n}]$ of degree $\mathbf a$ forces $f$ to have multidegree $\mathbf a - \sum_{i=1}^n \mathbf e_i$, so $f$ can be chosen to be monomial. So, $H_n(K^R)$ has a basis consisting of monomial cycles as well. Dao and De Stefani in \cite{dao-destefani-2020} prove that $H_2(K^R)$ always admits a basis of monomial cycles (they show this for $n = 3$, but the proof is general). We use the same method in Section \ref{proof mon basis} to show that for certain ideals, $H_p(K^R)$ admits a monomial cycle basis for all $p \geq 1$.

\subsection{Multiplication on Homology}
The Koszul complex $K^R$ can be viewed as the exterior algebra of $R^n$, with multiplication given by $(\bar f e_\sigma)(\bar g e_{\tau})$ $ =$ $ \overline{fg} e_{\sigma} \wedge e_{\tau}$. Furthermore, the multiplication satisfies the (graded) Leibniz rule, making $K^R$ a differential graded algebra. This implies that multiplication on $K^R$ descends to multiplication on $H_*(K^R)$. Notice that this product respects multidegree and homological degree. 
\par
For example, take $J$ as above. The homolgy classes $[\bar x_3 e_1]$ and $[\bar x_4 e_2]$ have multidegrees $(1,0,1,0)$ and $(0,1,0,1)$, respectively, and homological degree 1. Their product is $[\overline{x_3 x_4} e_{12}]$, which has multidegree $(1,1,1,1)$ and homological degree 2. Since $H_2(K^{S/J})_{(1,1,1,1)} = 0$, this product must be a boundary. Indeed, $\partial(\bar x_4 e_{123}) = \overline{x_3 x_4} e_{12}$. This can also be verified by Theorem \ref{main thm}, as $x_3 x_4 \in J + (x_3,x_4)^2$.
\par
The multiplicative structure in $H_*(K^R)$ informs us about the structure of the resolution of $k$ over $R$. The most well-studied relationship between the two is the Golod condition. We say $R$ (or $I$) is {\bf Golod} if 
\[
\sum_{j \geq 0} \dim_k \Tor_j^R(k,k) t^j \leq \frac{(1+t)^n}{1 - \sum_{j \geq 1} \dim_k \Tor_j^S(R,k) t^{j+1} }
\]
where $``\leq"$ here is taken coefficientwise. The vanishing of products in $H_{\geq 1}(K^R)$ is a necessary condition for Golodness. In four or fewer variables, it is sufficient as well (the three variable case is well-known and easily proved; for four variables see \cite{katthan}, Thoerem 6.3(4)). 
\par 
Thoerem \ref{main thm} implies conditions in terms of ideal inclusions which determine when products of monomial cycles vanish. To illustrate this, take $R = S/J$ as above. A product of monomial cycles $\bar u e_1$ and $\bar v e_2$ is again a monomial cycle $\overline{uv} e_{12}$. By Theorem \ref{main thm}, this product is non-zero if and only if $uv \in J + (x_3,x_4)^2$. Since $[J:x_1][J:x_2] = (x_3,x_4)^2$, the product is guaranteed to be the zero class in homology. In fact, $J$ satisfies all of the inclusions of Theorem \ref{four var}, so $J$ is Golod.

\section{Boundary Ideals}
\label{section mon cycles}
In this section we define boundary ideals. The definition relies on the circuits of a particular matroid. A {\bf matroid} $\mathcal{M}$ is a pair $(A,\mathcal{C})$, where $A$ is a set and $\mathcal{C}$ is a subset of $2^A$, called the set of {\bf circuits} of $\mathcal{M}$, satisfying
\begin{itemize}
\item $\emptyset \not\in \mathcal{C}$
\item For $C_1, C_2 \in \mathcal{C}$ and $C_1 \subset C_2$, then $C_1 = C_2$.
\item For $C_1,C_2 \in \mathcal{C}$   with $C_1 \neq C_2$ and $c \in C_1 \cap C_2$, then there is a $C_3 \in \mathcal{C}$ such that $C_3 \subset C_1 \cup C_2 \backslash \{c\}$. 
\end{itemize}
The prototypical example of a matroid is where $A$ is a collection of vectors in a vector space and $\mathcal{C}$ is the collection of minimal linearly dependent subsets.
\par 
Let $\binom{[n]}{p}$ denote the collection of subsets of $[n]$ of size $p$. For any $p < n$, we define a matrix $M(n,p)$ whose rows are indexed by $\binom{[n]}{p}$ and whose columns are indexed by $\binom{[n]}{p+1}$. The entry $M(n,p)_{\sigma,\tau}$ is $0$ if $\sigma \not\subset \tau$ and $\sgn(\tau,\tau \backslash \sigma)$ otherwise. We define $\mathcal{M}^n_p[k] = (\binom{[n]}{p},\mathcal{C}^n_p[k])$ to be the matroid whose circuits are the minimal sets of dependent rows of $M(n,p)$ over the field $k$.
\par
Fix a $\sigma$ in $\binom{[n]}{p}$ and a monomial ideal $I \subset k[x_1,\dots,x_n]$. For any set $A \subset [n]$, let $x_A := \prod_{a \in A} x_a$. For any set $\sigma' \in \binom{[n]}{p}$, let $I^\sigma_{\sigma'}$ denote the ideal $x_{\sigma' \backslash \sigma} [I:x_{\sigma \backslash \sigma'}]$. The {\bf boundary ideal} of $I$ with respect to $n$ and $\sigma$ is
\[
B^{n,\sigma}_I := \bigcap_{\substack{C \in \mathcal{C}^n_p[k] \\ \text{s.t. } \sigma \in C}} \sum_{\sigma' \in C} I^\sigma_{\sigma'}
\]
and we define $B^{n,[n]}_I$ to be $I$.

\begin{example}
\label{4var bound ideals}
We compute the boundary ideals in four variables from the definition.
\[
M(4,3) = 
\begin{blockarray}{cc}
 & 1234 \\
\begin{block}{c(c)}
123 & -1  \\
124 & 1   \\
134 & -1  \\
234 & 1  \\
\end{block}
\end{blockarray}
 \]
where $ijk\ell$ denotes the set $\{i,j,k,\ell\}$. The minimal sets of linearly dependent rows are just the two element sets. Thus,
\[
B^{4,\{1,2,3\}}_I = (I + x_4[I:x_3]) \cap (I + x_4[I:x_2]) \cap (I + x_4[I:x_1]) = I + x_4[I:(x_1,x_2,x_3)]
\]
\par
$M(4,2)$ is the matrix
\[
\begin{blockarray}{ccccc}
 & 123 & 124 & 134 & 234 \\
\begin{block}{c(cccc)}
12 & 1 & 1 & 0 & 0  \\
13 & -1 & 0 & 1 & 0  \\
14 & 0 & -1 & -1 & 0  \\
23 & 1 & 0 & 0 & 1  \\
24 & 0 & 1 & 0 & -1  \\
34 & 0 & 0 & 1 & 1 \\
\end{block}
\end{blockarray}
 \]
One can check that the minimal linearly dependent sets of rows which include the row $12$ are $\{12,13,14\}$, $\{12,23,24\}$, $\{12,13,24,34\}$, and $\{12,14,23,34\}$, regardless of the characteristic of $k$. The boundary ideal is therefore the intersection of four ideals:
\[
B^{4,\{1,2\}}_I = (I + x_3 [I:x_2] + x_4 [I:x_2]) \cap (I + x_3 [I:x_1] + x_4 [I:x_1]) \cap
\]
\[
 (I + x_3 [I: x_2] + x_4 [I:x_1] + x_3 x_4 [I:x_1 x_2]) \cap (I + x_4 [I: x_2] + x_3 [I:x_1] + x_3 x_4 [I: x_1 x_2])
\]
\[
= I + (x_3,x_4)[I:(x_1,x_2)] + (x_3x_4)[I:x_1x_2] \cap ( x_3[I:x_1] \cap x_4[I:x_2] + x_3[I:x_2] \cap x_4[I:x_1] )
\]
$M(4,1)$ is the matrix
\[
\begin{blockarray}{ccccccc}
 & 12 & 13 & 14 & 23 & 24 & 34 \\
\begin{block}{c(cccccc)}
1 & -1 & -1 & -1 & 0 & 0 & 0  \\
2 & 1 & 0 & 0 & -1 & -1 & 0  \\
3 & 0 & 1 & 0 & 1 & 0 & -1  \\
4 & 0 & 0 & 1 & 0 & 1 & 1  \\
\end{block}
\end{blockarray}
 \]
The only set of dependent rows is the set of all rows, so
\[
B^{4,\{1\}}_I = I + x_2[I:x_1] + x_3[I:x_1] + x_4[I:x_1] = I + (x_2,x_3,x_4)[I:x_1]
\]
\end{example}

\subsection{Full Simplicial Matroids and the Structure of Boundary Ideals}
The matroid $\mathcal{M}^n_p[k]$ is isomorphic to a well-studied matroid called the full simplicial matroid. We defer to \cite{cordovillindstrom} for background on full simplicial matroids. In this subsection, we will translate theorems about the structure of full simplicial matroids into facts about the structure of boundary ideals. We give explicit descriptions of $B_I^{n,\sigma}$ for $|\sigma| = 1,n-1$. Using facts about full simplicial matroids and (co)graphic matroids, we give explicit descriptions of $\mathcal{C}^n_p[k]$ for $p = 2,n-2$. These descriptions will not depend on the characteristic; however, when $3 \leq |\sigma| \leq n-3$, we show that $B^{n,\sigma}_I$ does depend on the characteristic. 
\par
The {\bf full simplicial matroid} $S^n_p[k]$ is the matroid with base set $\binom{[n]}{p}$ and circuits the minimal linearly dependent rows of (in our notation) $M(n,p-1)^t$  \cite[Proposition 6.2.5]{cordovillindstrom} \footnote{While it is standard to use columns for matroid representations, we choose to use rows for the sake of proving Theorem \ref{main thm}. }. 
\begin{proposition}
\label{isos}
$\mathcal{M}^n_p[k]$ is isomorphic to the matroid dual $S^n_p[k]^*$ and also to $S^n_{n-p}[k]$.
\end{proposition}

\begin{proof}
The circuits of the dual $S^n_p[k]^*$ are the minimal dependent rows of $M(n,p)$ \cite[Proposition 6.2.7]{cordovillindstrom}, which define $\mathcal{M}^n_p[k]$. Furthermore, $S^n_p[k]^*$ is isomorphic to $S^n_{n-p}[k]$ by taking complements of sets \cite[Proposition 11.4]{cr}.  
\end{proof}

\begin{proposition}
\[
B^{n,\{1\}}_I = I + (x_2,\dots,x_n)[I:x_1]
\]
\end{proposition}

\begin{proof}

$M(n,1)_{1,1i} = -1$, and, $M(n,1)_{j,1i} \neq 0$ if and only if $j = i$, in which case $M(n,1)_{i,1i} = 1$. So the only candidate for a set of linearly dependent rows is the entire set $\binom{[n]}{1}$. Indeed, each column of $M(n,1)$ has two non-zero entries, each with opposite signs, so the sum of all rows is zero. We conclude that the boundary ideal is the sum of ideals
\[
I + x_2[I:x_1] + \dots + x_n[I:x_1] = I + (x_2,\dots,x_n)[I:x_1]
\]
\end{proof}

\begin{proposition}
\[
B^{n,\{1,\dots,n-1\}}_I = I + x_n[I:(x_1,\dots,x_{n-1})]
\]
\end{proposition}

\begin{proof}
 The matrix $M(n,n-1)$ is a one-column matrix with a non-zero entry in each row. Thus, the circuits of $\mathcal{M}^n_{n-1}[k]$ are the two-element subsets of $\binom{[n]}{n-1}$. The boundary ideal is then the intersection of $n-1$ ideals
\[
B^{n,\{1,\dots,n-1\}}_I = \bigcap_{i = 1}^{n-1} (I + x_n [I : x_i]) = I + x_n [I:(x_1,\dots,x_{n-1})]
\]

\end{proof}

It is easy to prove that the circuits of $\mathcal{M}^n_p[k]$ containing $\pi(\sigma)$ for some permutation $\pi$ of $[n]$ are just the permutations under $\pi$ of the circuits containing $\sigma$. So, the boundary ideals for any $\sigma$ with  $|\sigma| =1,n-1$ can be calculated by permuting the indices of the previous propositions. 
\par
We can leverage Thoerem \ref{main thm} to show that certain ideals, which correspond to boundaries of monomials, are contained in every boundary ideal.

\begin{proposition}
\label{mon bound ideal}
For $|\sigma| < n$, $I + (x_j : j \not \in \sigma)[I:(x_i : i \in \sigma)] \subset B^{n,\sigma}_I$. 

\end{proposition}

\begin{proof}
Suppose $u$ is a monomial such that $\bar u e_{\sigma}$ is a monomial cycle and $u \in (x_j : j \not\in \sigma)[I:(x_i : i \in \sigma)]$.  Choosing $\ell \not \in \sigma$, we can look at the boundary of a ``monomial" element:
\[
\partial(\sgn(\sigma \cup \ell, \ell) \frac{\bar u}{x_\ell} e_{\sigma \cup \ell}) = \bar u e_{\sigma} + \sum_{i \in \sigma} \sgn(\sigma \cup \ell,\ell) \sgn(\sigma,i) \frac{\overline{x_i u}}{x_\ell} e_{\sigma \backslash i \cup \ell}
\]
By assumption, all but possibly the first term of the sum are zero in $R$. Regardless, $\bar u e_{\sigma}$ is a boundary and, by Theorem \ref{main thm}, $u \in B_I^{n,\sigma}$.
\end{proof}

\begin{example}
$B^{4,\sigma}_I$ with $|\sigma| = 2$ is the first boundary ideal where $I + (x_j : j \not \in \sigma)[I:(x_i : i \in \sigma)] \neq B^{n,\sigma}_I$. For example, when $I = (x_1 x_2, x_1 x_4, x_2 x_3)$, then $u = x_3 x_4$ is in $x_3 x_4[I:x_1 x_2] \cap x_3[I:x_1] \cap  x_4[I:x_2] \subset B^{4,\{1,2\}}_I$, so $\bar u e_{12}$ a boundary. An example of a preimage under the boundary map is $\bar x_4 e_{123} + \bar x_2 e_{134}$. However, $\partial(\bar x_4 e_{123}) \neq \overline{x_3 x_4} e_{12} \neq \partial(\bar x_3 e_{124})$, and $\bar u e_{12}$ can't be the boundary of any other monomial element in $K^R_3$ for multidegree reasons.
\end{example}

\begin{proposition}
Let $\mathcal{C}(K_n)$ denote the set of circuits (in the sense of graphs) of the complete graph on $n$ vertices. Then $\mathcal{C}^n_{n-2}[k] = \{ \{[n] \backslash e : e \in C\}: C \in \mathcal{C}(K_n)\}$.  
\end{proposition}

\begin{proof}
By Proposition \ref{isos}, $\mathcal{M}^n_{n-2}[k]$ is isomorphic to $S^n_2[k]$ by taking complements. Now $S^n_2[k]$ is a graphic matroid (with $M(n,1)^t$ the incidence matrix) corresponding to the complete graph $K_n$. The circuits of graphic matroids are simply the circuits (in the graph sense) of the graph.
\end{proof}

\begin{proposition}
 $\mathcal{C}^n_2[k]$ is the set of bonds of the complete graph $K_n$. That is, $\mathcal{C}^n_2[k]$ is the set of all $E(A,B)$ for $A,B$ a partition of $[n]$, where $E(A,B)$ is the set of all edges with one vertex in $A$ and the other in $B$.
\end{proposition}

\begin{proof}
Since $\mathcal{M}^n_2[k] \cong S^n_2[k]^*$, it is the dual of a graphic matroid. The circuits of the dual matroid are the bonds of the original matroid. For graphic matroids, the bonds are the bonds of the graph, i.e. the minimal cutting sets. Cutting sets are the sets of edges which, when deleted, disconnect the set. Since $K_n$ is connected, the minimal cutting sets, when removed, divide $K_n$ into two disconnected sets of vertices $A$ and $B$. To do this minimally, we have to remove exactly the edges in $E(A,B)$. 

\end{proof}

\begin{example}
We can compute $B^{4,\{1,2\}}_I$ in two ways. Consider the complete graph $K_4$. The cycles containing $34$ are $\{13,14,34\}$, $\{23,24,34\}$, $\{12,13,24,34\}$, and $\{12,14,23,34\}$. Taking the complements of each edge in the cycle gives us the collections $\{12,23,24\}$, $\{12,13,14\}$, $\{12,13,24,34\}$, and $\{12,14,23,34\}$, which are the circuits we found in the computation of $B^{4,\{1,2\}}_I$. These are also the bonds of $K_4$ containing $12$, being the cutting sets $E(2,134),E(1,234),E(14,23)$, and $E(13,24)$, respectively. The cutting sets can also be viewed as the 2-color graphs on four vertices, with 1 and 2 being distinct colors.

\end{example}

\begin{proposition}
$\mathcal{C}^{n}_p[k]$ is characteristic independent for $p = 1,2,n-2,n-1$ or $n$, but depends on the characteristic of $k$ for $3 \leq p \leq n-3$.

\end{proposition}

\begin{proof}
Cordovil and Las Vergnas \cite{clv} and Lindstr\"om \cite{lind} showed that $S^n_p[k]$ is representable over every field if and only if $p \leq 2$ or $p \geq n-2$, and that it is representable over a field of characteristic 2 if and only if the characteristic of $k$ is 2. Cordovil \cite{cordovil} showed that if the full simplicial matroid was representable over fields $k$ and $k'$, then $S^n_p[k] \cong S^n_p[k']$.

\end{proof}

\begin{example}
In $M(6,3)$, the sum of rows corresponding to $123$, $124$, $135$, $146$, $156$, $236$, $245$, $256$, $345$, and $346$, the minimal non-faces of a triangulation of the real projective plane, is twice the row corresponding to $456$. The sum is then 0 in characteristic 2, but not otherwise. One can check (as in the proof of Theorem \ref{main thm}) that when $I$ is the Stanley-Reisner ring associated to this triangulation of the real projective plane, $[\overline{x_4 x_5 x_6} e_{123}] \neq 0$ when $\text{char } k = 2$ but is zero otherwise.

\end{example}

\subsection{Proof of Theorem \ref{main thm}}

We rephrase Theorem \ref{main thm} in terms of our present notation. 

\begin{theorem*}
    Let $u$ be a monomial of $S$ and $\sigma \in \binom{[n]}{p}$ such that $u \not \in I$ and $u \in I:(x_j : j \in \sigma)$ (that is, $\bar u e_\sigma$ is a non-zero monomial cycle in $K^R_p$). Then $[\bar u e_\sigma]$ is the zero class in Koszul homology if and only if $u \in B_I^{n,\sigma}$
\end{theorem*}

\begin{proof}
    When $\sigma = [n]$, there are no boundaries in $K^R_n$, so $\bar u e_\sigma$ is the trivial class in homology if and only if $u \in I = B_I^{n,[n]}$. 
    Let $1 \leq p < n$. Since $\partial$ preserves multidegree, $\bar u e_\sigma$ is a boundary if and only if it is the image under the Koszul differential of some element of the same multidegree. That is, $\bar u e_\sigma$ is a boundary if and only if
    \[
    \bar u e_\sigma = \partial( \sum_{\tau \in \binom{[n]}{p+1}} c_\sigma \overline{\frac{x_\sigma u}{x_\tau}} e_\tau ) = \sum_{\sigma' \in \binom{[n]}{p}} (\sum_{j \not\in \sigma'} \sgn(\sigma' \cup j, j) c_{\sigma' \cup j})\overline{\frac{x_\sigma u}{x_{\sigma'}}} e_{\sigma'}
    \]
    where $c_\tau \in k$ are coefficients, with the understanding that $c_\tau = 0$ if $x_\tau$ does not divide $x_\sigma u$. Thus, $\bar u e_\sigma$ is a boundary if and only if the following conditions hold
    \begin{itemize}
    \item
    $\sum_{j \not \in \sigma} \sgn(\sigma \cup j, j) c_{\sigma \cup j} = 1$.
    \item
    For all $\sigma' \neq \sigma$, $\sum_{j \not \in \sigma'} \sgn(\sigma' \cup j, j) c_{\sigma' \cup j} = 0$ or $\frac{x_\sigma u}{x_\sigma'} \in I$, i.e. $ u \in I^\sigma_{\sigma'}$.
    \end{itemize}
     These conditions are equivalent to the following linear system being consistent:
    \begin{itemize}
    \item
    $\sum_{j \not \in \sigma} \sgn(\sigma \cup j, j) c_{\sigma \cup j} = 1$.
    \item
    $\sum_{j \not \in \sigma'} \sgn(\sigma' \cup j, j) c_{\sigma' \cup j} = 0$ for all $\sigma' \neq \sigma$ such that $u \not \in I^{\sigma}_{\sigma'}$.
    \end{itemize}
    Now these linear equations correspond to rows of $M(n,p)$. Let $M^u(n,p)$ denote the submatrix of $M(n,p)$ which includes only the rows corresponding to $\sigma$ and to $\sigma' \neq \sigma$ such that $u \not \in I^{\sigma}_{\sigma'}$. Let $\widetilde{M}^u(n,p)$ be $M^u(n,p)$ with the row $v_\sigma$ corresponding to $\sigma$ removed. The linear system is consistent, then, if and only if the rank of $M^u(n,p)$ is the rank of the augmented matrix
\[
\left(\begin{array}{c|c}  
 v_\sigma & 1 \\  
 \widetilde{M}^u(n,p) & 0  
\end{array}\right) 
\] 
Now the rank of the augmented matrix is $1 + \rk(\widetilde{M}^u(n,p))$, and $\rk(M^u(n,p)) = 1 + \rk(\widetilde{M}^u(n,p)))$ if and only if $v_\sigma$ is linearly independent of the rows of $\widetilde{M}^u(n,p)$. Alternatively, the system is inconsistent if and only if $M^u(n,p)$ has a (minimal) set $C$ of linearly dependent rows containing $v_\sigma$. Now $C$ is a subset of rows of $M^u(n,p)$ if and only if $u \not \in \sum_{\sigma' \in C} I^\sigma_{\sigma'}$. So, $\bar u e_\sigma$ is a boundary if and only if $u \in \sum_{\sigma' \in C} I^\sigma_{\sigma'}$ for all possible $C$. But these $C$ are simply the circuits of $\mathcal{M}^n_p[k]$ containing $\sigma$, which concludes the proof.
\end{proof}

\section{Products of Monomial Cycles}
\label{section four vars}

A product of monomial cycles is again a monomial cycle. We can therefore use Theorem \ref{main thm} to derive Theorem \ref{necessary cond}, which establish conditions in terms of ideal inclusions that determine when products of monomial cycles vanish. These are new necessary conditions for monomial ideals to be Golod. We then show that these conditions are in fact sufficient for Golodness in four variables.

\subsection{Necessary Conditions for Golodness}

\begin{proposition}
\label{product condition}
Let $2 \leq p \leq n$ and $1 \leq q < p$. Products of monomial cycles in $H_q(K^R)$ and $H_{p-q}(K^R)$, respectively are trivial if and only if the following inclusion of ideals holds for all permutations of the indices:
\[
[I:(x_1,\dots,x_q)][I:(x_{q+1},\dots,x_{p})] \subset B^{n,\{1,\dots,p\}}_I
\]
\end{proposition}

\begin{proof}
    Suppose the inclusion fails for some permutation $\pi$ of the indices. Let $\sigma = \{\pi(1),\dots,\pi(q)\}$ and $\tau = \{\pi(q+1),\dots,\pi(p)\}$ and let $u,v$ be monomials in $I:(x_\ell : \ell \in \sigma)$ and $I:(x_m: m \in \tau)$, respectively, such that $uv \not \in B_I^{n,\sigma \cup \tau}$. Then $\bar u e_\sigma$ and $\bar v e_\tau$ are monomial cycles, and their product $\overline{uv} e_\sigma \wedge e_\tau$ is a monomial cycle in $Z_{p}(K^R)$ which, by Theorem 1, is not a boundary. So, $[\overline{uv} e_\sigma \wedge e_\tau]$ is a nontrivial homology class.
    \par
    Suppose all of the ideal inclusions hold. Then for monomial cycles $\bar u e_\sigma \in Z_q(K^R)$ and $\bar v e_\tau \in Z_{p-q}(K^R)$, their product is the monomial cycle $\overline{uv} e_\sigma \wedge e_\tau$. Since $uv \in B_I^{n,\sigma \cup \tau}$, Theorem \ref{main thm} implies that $[\overline{uv} e_\sigma \wedge e_\tau]$ is the trivial homology class. 
\end{proof}

As a corollary to Proposition \ref{product condition}, we obtain Theorem \ref{necessary cond}. As further corollaries, we obtain necessary conditions for Golodness. To write these corollaries we first record some definitions. We say $I$ {\bf admits a monomial basis in Koszul homology} if, for all $1 \leq p \leq n$, either $H_p(K^R)$ has a basis consisting of monomial cycles or $H_p(K^R) = 0$. We listed some examples of such ideals in the introduction, and we will find more in Section \ref{section mon basis}. An ideal $I$ is {\bf weakly Golod} if all products in Koszul homology vanish. 

\begin{corollary}
Let $I$ be a monomial ideal that admits a monomial cycle basis in Koszul homology. Then $I$ is weakly Golod if and only if
\[
[I:(x_1,\dots,x_q)][I:(x_{q+1},\dots,x_p)] \subset B_I^{n,\{1,2,\dots,p\}}
\]
holds for all $2 \leq p \leq n$ and for all $1 \leq q < p$ and for all permutations of the indices.
\end{corollary}

If all products in Koszul homology vanish, then certainly all products of monomial cycles do. We therefore obtain the following necessary conditions for vanishing of products in Koszul homology. 

\begin{corollary}
If the product $H_q(K^R) \otimes H_{p-q}(K^R) \to H_{p}(K^R)$ is trivial, then the following inclusion of ideals holds for all permutations of the indices:
\[
[I:(x_1,\dots,x_q)][I:(x_{q+1},\dots,x_{p})] \subset B_I^{n,\{1,\dots,p\}}
\]
\end{corollary}

\begin{corollary}
Let $2 \leq p \leq n$ and $1 \leq q < p$. If products of Koszul homology classes into $H_p(K^R)$ are trivial, then the following holds for all permutations of the indices:
\[
[I:(x_1,\dots,x_q)][I:(x_{q+1},\dots,x_p)] \subset B_I^{n,\{1,2,\dots,p\}}
\]

\end{corollary}

\begin{corollary}
\label{weak golod cond}
Let $I$ be a (weakly) Golod monomial ideal. Then the following inclusions hold for all $2 \leq p \leq n$ and for all $1 \leq q < p$ and for all permutations of the indices:
\[
[I:(x_1,\dots,x_q)][I:(x_{q+1},\dots,x_p)] \subset B_I^{n,\{1,2,\dots,p\}}
\]

\end{corollary}

\par

\subsection{Golod Monomial Ideals in Four Variables}

Dao and De Stefani showed in \cite{dao-destefani-2020} that monomial ideals in three or fewer variables admit a monomial basis in Koszul homology. It is easy to show, as we did in Section \ref{section prelim}, that any monomial ideal $I$ admits a monomial basis in Koszul homology in degrees $1$ and $n$; the authors showed that degree $2$ admits such a basis as well. With this, they classify Golod monomial ideals in terms of ideal inclusions, equivalent to those in Theorem \ref{necessary cond}. We show that having vanishing products of monomial cycles is sufficient for being Golod in four variables, thus classifying Golod four variable monomial ideals in terms of the inclusions Theorem \ref{necessary cond}. We show this despite the fact that not all monomial ideals in four variables admit a monomial basis in Koszul homology in degree three, as Example \ref{eg} showed.

\par
When $n=4$,  $H_1(K^R),H_2(K^R)$, and $H_4(K^R)$ admit a monomial cycle basis, so most products in Koszul homology can be considered as products of monomial cycles. The only exception is $H_1(K^R) \otimes H_3(K^R) \to H_4(K^R)$ where the homology class in $H_3(K^R)$ may not be represented by a monomial cycle. It turns out that having vanishing monomial products is still sufficient to determine vanishing in this case. We show this by exploiting multidegrees.

\begin{proposition}
\label{h1h3 vanish}
Let $I$ be a monomial ideal in $k[x_1,x_2,x_3,x_4]$. The product $H_1(K^R) \otimes H_3(K^R) \to H_4(K^R)$ is trivial if and only if the inclusion
\[
[I:x_1][I:(x_2,x_3,x_4)] \subset I
\]
holds under all permutations of the indices.
\end{proposition}

\begin{proof}
    If one of the inclusions doesn't hold, then we can find a non-trivial product as in the proof of Proposition \ref{product condition}. So, assume the inclusions do hold. 
    \par
    We can choose $[\bar u e_i]$ to be an element of a monomial cycle basis of $H_1(K^R)$. Let $[c_1 \bar v_1 e_{234} + c_2 \bar v_2 e_{134} + c_3 \bar v_3 e_{124} + c_4 \bar v_4 e_{123}]$ be a multigraded basis element in $H_3(K^R)_{\mathbf a}$, where the $c_j$ are coefficients in $k$. The product of the two basis elements is $[\pm c_4 \overline{u v_i} e_{1234}]$, which is a monomial cycle up to scaling. Since $\bar u e_i$ is a cycle, $x_i u \in I$. So, if $x_i$ divides $ v_i$, the product vanishes. If $x_i$ does not divide $v_i$, then the multidegree of $c_i \bar v_i e_{[4]\backslash \{i\}}$ has 0 in the $i$th place, so $\mathbf a$ has 0 in the $i$th place as well. But the other $e_\tau$ for $\tau \neq [4] \backslash \{i\}$ have multidegrees which are positive in the $i$th place. So, we can assume $c_j = 0$ for $j \neq i$. But then $c_i \bar v_i e_{[4] \backslash \{i\}}$ is a monomial cycle, so $v_i \in I:(x_j : j \neq i)$ and, by assumption, the product vanishes.
\end{proof}

With this, we can prove Theorem \ref{four var}, which we restate here for convenience.

\begin{theorem*}
Let $I \subset (x_1,\dots,x_4)^2$ be a monomial ideal in $k[x_1,x_2,x_3,x_4]$. Then $I$ is Golod if and only if the following hold under all permutations of the indices:
    \[
    [I:x_1][I:(x_2,x_3,x_4)] \subset I
    \]
    \[
    [I:(x_1,x_2)][I:(x_3,x_4)] \subset I
    \]
    \[
    [I:x_1][I:(x_2,x_3)] \subset I + x_4[I:(x_1,x_2,x_3)] 
    \]
    \[
    [I:x_1][I:x_2] \subset I + (x_3,x_4)[I:(x_1,x_2)] + (x_3x_4)[I:x_1x_2] \cap ( x_3[I:x_1] \cap x_4[I:x_2] + x_3[I:x_2] \cap x_4[I:x_1] )
    \]
\end{theorem*}

\begin{proof}
These inclusions are exactly the ones given Theorem \ref{necessary cond} when $n=4$, given our calculation of the boundary ideals for $n=4$ in Example \ref{4var bound ideals}. They are then equivalent to the vanishing of products of monomial cycles in Koszul homology. Proposition \ref{h1h3 vanish} and the first inclusion cover the only other possible non-trivial products. 
\end{proof}

Having vanishing monomial products is not sufficient for Golodness in five variables. In fact, it is not even sufficient for weak Golodness, as a modified version of Example \ref{eg} shows. 

\begin{example}
Let $I = (x_1 x_3,x_1 x_4,x_2 x_3,x_2 x_4,x_5^2) \subset k[x_1,\dots,x_5]$. The following product of cycles
\[
(\bar x_5 e_5) \wedge (\bar x_1 e_{234} - \bar x_2 e_{134}) = \overline{x_2 x_5} e_{1345} - \overline{x_1 x_5} e_{2345} 
\]
is not a boundary. If it was the boundary of an element of the same multidegree, then we would have
\[
\partial( \pm \bar x_5 e_{12345}) = \overline{ x_2 x_5} e_{1345} - \overline{x_1 x_5} e_{2345} 
\]
which is not the case. Furthermore, $[\overline{x_2 x_5} e_{1345} - \overline{x_1 x_5} e_{2345}] = [\overline{x_3 x_5} e_{1245} - \overline{x_4 x_5} e_{1235}]$ cannot be written as a monomial cycle. Moreover, a Macaulay2 computation \cite{M2} shows that the fourth total Betti number of $I$ is 1. Thus, any product of monomial cycles into $H_4(K^R)$ must vanish (otherwise the above product could be chosen to be a monomial cycle). We conclude that the multiplication $H_1(K^R) \otimes H_3(K^R) \to H_4(K^R)$ is non-trivial, while the multiplication of monomial cycles is trivial.

\end{example}

\par
Even if all products in Koszul homology vanish in five variables, $I$ may not be Golod. Katthan \cite{katthan} gave an example of a weakly Golod ideal $I$ in five variables which is not Golod. We can conclude that this $I$ satisfies the conditions of Theorem \ref{necessary cond}, but is not Golod.

\section{Symmetric Monomial Ideals}
\label{section symm mon}
In this section, we apply Theorem \ref{necessary cond} to symmetric monomial ideals. We establish nice conditions in terms of the generators classifying vanishing of products of monomial cycles. We conjecture that these conditions in fact classify Golodness as well.
\par
Recall that a symmetric monomial ideal is a monomial ideal that is stable under all permutations of the variables. For example, the ideal $I = (x_1^2x_2, x_1^2 x_3, x_1x_2^2,x_2^2x_3,x_1x_3^2,x_2x_3^2,x_1x_2x_3)$ is a symmetric monomial ideal in $k[x_1,x_2,x_3]$ with generating partitions $\Lambda(I) = \{(2,1,0),(1,1,1)\}$.

\subsection{Vanishing Products of Monomial Cycles}
\par 
Checking the conditions of Theorem \ref{necessary cond} for a symmetric monomial ideal $I$ is simpler than for a general monomial ideal, as we don't need bother checking the ideal inclusions under permutations of variables. We can make even further simplifications. In the end, it suffices to show only $n-2$ ideal inclusions:

\begin{theorem}
\label{symm conditions}
Let $I$ be a symmetric monomial ideal. Then all products of monomial cycles in Koszul homology vanish if and only if
\[
[I:(x_1,\dots,x_{\lfloor p/2 \rfloor})][I:(x_{\lfloor p/2 \rfloor +1},\dots,x_p)] \subset I + (x_{p+1},\dots,x_n)[I:(x_1,\dots,x_p)] 
\]
for all $2 \leq p \leq n$.

\end{theorem}

We can translate these inclusion conditions into conditions on $\Lambda(I)$. For $p \geq 2$, we say that $I$ satisfies $V_p$ if one of the following holds:

\begin{enumerate}
\item[$V_p^{(1)}$:] For all partitions $\lambda \in \Lambda(I)$ of length at most $\lfloor p/2 \rfloor$ with minimal first term among partitions in $\Lambda(I)$ of length at most $\lceil p/2 \rceil$, $\lambda_1 = \lambda_2$.

\item[$V_p^{(2)}$:] There exists a $\mu \in \Lambda(I)$ with $\lceil p/2 \rceil < \ell(\mu) \leq p$ such that $\mu_1$ smaller than the minimal first term among partitions of length at most $\lceil p/2 \rceil$ in $\Lambda(I)$.

\end{enumerate}

The following is a restatement of Theorem \ref{partition conditions} with this notation.
\begin{theorem*}
Let $I$ be a symmetric monomial ideal. Then all products of monomial cycles in Koszul homology vanish if and only if $I$ satisfies $V_p$ for all $2 \leq p \leq n$. 

\end{theorem*}

\begin{example}

In the course of proving Theorem \ref{partition conditions}, we will show that $V_p$ is equivalent to vanishing of products of monomial cycles into homological degree $p$. Let us translate of $V_2,V_3,V_4$ which, by Theorem \ref{four var}, classify Golod symmetric monomial ideals in four or fewer variables.
\par
\begin{enumerate}
\item[$V_2^{(1)}$:] For all partitions $\lambda$ of length 1 in $\Lambda(I)$, $\lambda_1 = \lambda_2$ in $\Lambda(I)$ (equivalently, there no partitions of length 1 in $\Lambda(I)$).
\item[$V_2^{(2)}$:] There is a $\mu \in \Lambda(I)$ of length 2 in $\Lambda(I)$.
\item[$V_3^{(1)}$:] For all partitions $\lambda$ of length 1 with minimal first term among partitions of length at most 2 in $\Lambda(I)$, $\lambda_1 = \lambda_2$ (equivalently, either there is no partition of length 1 or there is a partition of length 2 in $\Lambda(I)$).
\item[$V_3^{(2)}$:] There is a $\mu \in \Lambda(I)$ of length 3 with $\mu_1$ smaller than the minimial first term among partitions of length at most $2$ in $\Lambda(I)$.
\item[$V_4^{(1)}$:] For all partitions $\lambda$ of length at most 2 with minimal first term among partitions of length at most 2 in $\Lambda(I)$, $\lambda_1 = \lambda_2$.
\item[$V_4^{(2)}$:] There is a $\mu \in \Lambda(I)$ of length 3 or 4 with $\mu_1$ smaller than the minimial first term among partitions of length at most $2$ in $\Lambda(I)$.

\end{enumerate}

We can see that $V_2$ implies $V_3$, as both $V_2^{(1)}$ and $V_2^{(2)}$ imply $V_3^{(1)}$. Most other combinations of $V_p$ for $p = 2,3,4$ are independent, as the following table of examples in four variables shows:

\begin{center}
    
    \begin{tabular}{|c|c|}
        \hline 
        $\Lambda(I)$ &  $p$ such that $I$ is $V_p$\\ \hline \hline
        $(3,1,0,0),(2,2,0,0)$ & 2,3,4 - Golod \\ \hline
        $(3,0,0,0),(2,1,0,0)$ & 2,3 \\ \hline
        $(3,0,0,0),(2,1,1,0)$ & 3,4 \\ \hline
        $(2,0,0,0),(1,1,1,1)$ & 4 \\ \hline
        $(2,0,0,0)$ & none \\ \hline
    \end{tabular}
    
\end{center}

\end{example}
There are no symmetric monomial ideals in four variables which only satisfy $V_3$. If $I$ satisfies $V_3$ but not $V_2$, then there are partitions of legnths 1 and 3 in $\Lambda(I)$, but not of length 2, which implies $V_4^{(2)}$. 

\par
The $V_p$ are necessary conditions for $I$ to be Golod. One interesting class of Golod symmetric monomial ideals are the symmetric shifted ideals (see Section \ref{section mon basis} for more details). The ideal $I$ with $\Lambda(I) = \{(3,1,0,0),(2,2,0,0)\}$ from Example \ref{nonvanishing prod} shows that not all Golod symmetric monomial ideals are symmetric shifted.
\par

The reductions we use to prove Theorems \ref{symm conditions} and \ref{partition conditions} rely on the following lemma.

\begin{lemma}
\label{lemma c}
Let $I$ be a symmetric monomial ideal. Let $p$ be an integer with $2 \leq p \leq n$ and $q$ be an integer with $1 \leq q \leq \lfloor p/2 \rfloor$. Suppose that $f \in I:(x_1,\dots,x_q)$ and $g \in  I:(x_{q+1},\dots,x_p)$ be monomials such that $fg \not\in I +(x_{p+1},\dots,x_n)[I:(x_1,\dots,x_p)]$. Then there exists a $c$ such that $f = x_1^c \dots x_q^c$ and $g = x_{q+1}^c \dots x_p^c$. Furthermore, $c = \min \{\lambda_1: \lambda \in \Lambda(I) \text{ and } \ell(\lambda) \leq p-q\} - 1$.
\end{lemma}

\begin{proof}
Since $fg \not\in I$, we can write $f = x_1^{a_1} \dots x_q^{a_q} x_{p+1}^{\alpha_{p+1}} \dots x_n^{\alpha_n}$ and $g = x_{q+1}^{b_{q+1}} \dots x_p^{b_p} x_{p+1}^{\beta_{p+1}} \dots x_n^{\beta_n}$.
\par
Suppose that $a_i > b_j$ for some $i$ and $j$. Then, applying the permutation $(i,j)$ to $fg$, we get
\[
x_j^{a_i-b_j} g (\frac{f}{x_i^{a_i - b_j}})
\]
which is in $I$, as $x_j g \in I$. Since $fg \not \in I$, we conclude that $a_i \leq b_j$ for all $i$ and $j$. A similar argument yields the other inequality, so all $a_i$ and $b_j$ are the same constant $c$.
\par
Next, we show that all $\alpha_k = \beta_k = 0$ for $k > p$. Since $fg \not \in (x_{p+1},\dots,x_n)[I:(x_1,\dots,x_p)]$ we know that for all $k > p$, $fg \not \in x_k[I:(x_1,\dots,x_p)]$. This means that either $\alpha_k = \beta_k = 0$ or 
\[
x_i \frac{fg}{x_k} = x_1^c \dots x_i^{c+1} \dots x_p^c x_{p+1}^{\alpha_{p+1} + \beta_{p+1}} \dots x_{k}^{\alpha_{k} + \beta_{k}-1} \dots  x_{n}^{\alpha_{n} + \beta_{n}}
\]
is not in $I$ for some $i, 1 \leq i \leq p$. The latter case is equivalent to $x_i \frac{fg}{x_k} \not\in I$ for \textit{all} $i, 1 \leq i \leq p$ since $I$ is symmetric. In particular, $x_1 fg \not \in I$, so $\beta_k = 0$ (lest $x_1 f$, which is in $I$, divide $x_1 \frac{fg}{x_k}$). Since $x_p \frac{fg}{x_k} \not\in I$, we similarly conclude $\alpha_k = 0$. In either case, $\alpha_k = \beta_k = 0$. This proves the first statement.
\par
We now prove the last statement. Let $c' = \min \{\lambda_1 : \lambda \in \Lambda(I) \text{ and } \ell(\lambda) \leq p-q \} - 1$. Suppose $c' < c$. Then there exists a $\mu \in \Lambda(I)$ of length at most $p-q$ such that $\mu_i \leq c$ for all $i$.  But then $x_{q+1}^{\mu_1} \dots x_p^{\mu_{p-q}}$, which is in $I$, divides $g = x_{q+1}^c \dots x_q^c$, which is a contradiction. If $c < c'$ then $c + 1 < \lambda_1$ for all $\lambda \in \Lambda(I)$ of length at most $p-q$. But then $x_1 f = x_1^{c+1} x_2^c \dots x_q^c$ is not in $I$, so $f \not\in I:(x_1,\dots,x_q)$, which is another contradiction. We conclude that $c = c'$. 

\end{proof} 

\begin{proposition}
\label{reduce boundary ideal}
Let $I$ be a symmetric monomial ideal, $2 \leq p \leq n$ and $1 \leq q \leq \lfloor p/2 \rfloor$. Then
\[
[I:(x_1,\dots,x_q)][I:(x_{q+1},\dots,x_p)] \subset B^{n,[p]}_I
\]
 if and only if
\[
[I:(x_1,\dots,x_q)][I:(x_{q+1},\dots,x_p)] \subset  I +(x_{p+1},\dots,x_n)[I:(x_1,\dots,x_p)]
\]

\end{proposition}

\begin{proof}
By Proposition \ref{mon bound ideal}, $(x_{p+1},\dots,x_n)[I:(x_1,\dots,x_p)] \subset B^{n,[p]}_I$. It is enough, then, to check that if $f$ and $g$ are defined as in Lemma \ref{lemma c}, then $[\bar f e_{1 \dots q}][\bar g e_{q+1 \dots p}]$ is not a boundary. 
\par
By Lemma \ref{lemma c}, $fg = x_1^c \dots x_p^c$ for some $c$. So, $[\bar f e_{1 \dots q}][\bar g e_{q+1 \dots p}] = [\overline{x_1^c \dots x_p^c} e_{1 \dots p}]$.  Now $x_1^c \dots x_p^c \not \in I$ by assumption. $\overline{x_1^c \dots x_p^c} e_{1 \dots p}$ is not a boundary, because its multidegree is non-zero in exactly $p$ places. But multihomogeneous elements of $K^R_{p+1}$ with non-zero multidegrees have multidegrees which are non-zero in at least $p+1$ places, so summands of their images under $\partial$ will as well.

\end{proof}

Proposition \ref{reduce boundary ideal} shows that if a product of monomial cycles is non-zero, it must have a precise form. See Example \ref{nonvanishing prod} for a concrete example. Futhermore, if a product of monomial cycles is a boundary, it is the boundary of a monomial element. 
\par
We now establish a condition on the generators of $I$ determining monomial product vanishing:

\begin{proposition}
\label{equivalence}
Let $I,p,q$ as above. Then
\[
[I:(x_1,\dots,x_q)][I:(x_{q+1},\dots,x_p)] \subset  I +(x_{p+1},\dots,x_n)[I:(x_1,\dots,x_p)]
\]
if and only if one of the following hold:

\begin{itemize}
    \item [(a)] For all partitions $\lambda \in \Lambda(I)$ of length at most $q$ with first term minimal among partitions in of length at most $p-q$ in $\Lambda(I)$, $\lambda_1 = \lambda_2$.
    
    \item [(b)] There exists a $\mu \in \Lambda(I)$ with $p-q < \ell(\mu) \leq p$ and $\mu_1$ smaller than the minimal first term among partitions of length at most $p-q$ in $\Lambda(I)$.
\end{itemize}

\end{proposition}

\begin{proof}
We will show that the negations are equivalent, namely,
\[
[I:(x_1,\dots,x_q)][I:(x_{q+1},\dots,x_p)] \not\subset  I +(x_{p+1},\dots,x_n)[I:(x_1,\dots,x_p)]
\]
if and only if both of the following hold
\begin{enumerate}

\item There exists a $\lambda \in \Lambda(I)$ of length at most $q$ with first term minimal among partitions of length at most $p-q$ in $\Lambda(I)$ such that $\lambda_1 > \lambda_2$.

\item For all $\mu \in \Lambda(I)$ with $p-q < \ell(\mu) \leq p$, we have $\mu_1 \geq \min\{\lambda_1: \lambda \in \Lambda(I) \text{ and } \ell(\lambda) \leq p-q\}$.

\end{enumerate} 

For the forward direction, we apply Lemma \ref{lemma c} to say that $f = x_1^c \dots x_q^c \in I:(x_1,\dots,x_q)$, $g = x_{q+1}^c \dots x_p^c \in I:(x_{q+1},\dots,x_p)$, and $fg = x_1^c \dots x_p^c \not\in I$, where $c + 1 = \min\{\lambda_1: \lambda \in \Lambda(I) \text{ and } \ell(\lambda) \leq p-q\}$. Since $x_1 f$ and $x_{q+1} g$ are in $I$, there is a partition $\lambda$ of length at most $q$ such that $x^\lambda$ divides $x_1 f$. A permutation of the same $x^\lambda$ will $x_{q+1} g$. Since $f$ and $g$ are not in $I$, we must have $\lambda_1 = c+1$ and $\lambda_1 > \lambda_2$. This proves 1.
\par
Since $x_1^c \dots x_p^c$ is not in $I$, we know that for any partition $\mu \in \Lambda(I)$ with $p-q< \ell(\mu) \leq p$, we have $\mu_1 \geq c+1$, lest $x^\mu$ divide $x_1^c \dots x_p^c$. This proves 2.
\par 
For the backward direction, assume 1 and 2 hold. Let $c = \min\{\lambda_1: \lambda \in \Lambda(I) \text{ and } \ell(\lambda) \leq p-q\} - 1$, $f := x_1^c \dots x_q^c$, and $g := x_{q+1}^c \dots x_p^c$. By 1, both $f \in I:(x_1,\dots,x_q)$ and $g \in I:(x_{q+1},\dots,x_p)$. By 2, $fg$ is not in $I$. Since $x_j$ does not divide $fg$ for $j > p$, $fg$ is not in $(x_{p+1},\dots,x_n)[I:(x_1,\dots,x_p)]$.

\end{proof}

\begin{proposition}
Let $I,p,q$ as above. When $q \geq 2$, 
\[
[I:(x_1,\dots,x_q)][I:(x_{q+1},\dots,x_p)] \subset  I +(x_{p+1},\dots,x_n)[I:(x_1,\dots,x_p)]
\]
implies
\[
[I:(x_1,\dots,x_{q-1})][I:(x_q,\dots,x_p)] \subset  I +(x_{p+1},\dots,x_n)[I:(x_1,\dots,x_p)]
\]
\end{proposition}

\begin{proof}
We prove the contrapositive using Proposition \ref{equivalence}. Suppose there is a $\lambda \in \Lambda(I)$ with $\ell(\lambda) \leq q-1$ and $\lambda_1$ minimal among partitions in $\Lambda(I)$ of length at most $p-q+1$ such that $\lambda_1 > \lambda_2$. Then certainly $\lambda_1$ is minimal among partitions of length at most $p-q$. Suppose that for all $\mu \in \Lambda(I)$ with $p-q+1 < \ell(\mu) \leq p$, we have
\[
\mu_1 \geq \min\{\lambda_1: \lambda \in \Lambda(I) \text{ and } \ell(\lambda) \leq p-q+1\} \geq  \min\{\lambda_1: \lambda \in \Lambda(I) \text{ and } \ell(\lambda) \leq p-q+1\}
\]
In fact, the second inequality above is an equality; the $\lambda$ found above achieves the minimum in both sets. So, it remains to check that partitions $\mu'$ of length $p-q+1$ have $\mu_1' \geq \{\lambda_1: \lambda \in \Lambda(I) \text{ and } \ell(\lambda) \leq p-q+1\}$. But this is true because $\mu'$ is in the set. 

\end{proof}

\begin{proof}[Proof of Theorems \ref{symm conditions} and \ref{partition conditions}]
By Proposition \ref{reduce boundary ideal}, we reduce the monomial product vanishing conditions to checking
\[
[I:(x_1,\dots,x_q)][I:(x_{q+1},\dots,x_p)] \subset  I +(x_{p+1},\dots,x_n)[I:(x_1,\dots,x_p)]
\]
for all $2 \leq p \leq n$ and $1 \leq q \leq \lfloor p/2 \rfloor$. The previous proposition reduces this to checking only $q = \lfloor p/2 \rfloor$. This proves Theorem \ref{symm conditions}.
Applying Proposition \ref{equivalence} to Theorem \ref{symm conditions} yields Theorem \ref{partition conditions}.
\end{proof}

\subsection{Golod Symmetric Monomial Ideals}

Theorem \ref{partition conditions} says that the symmetric monomial ideals satisfying $V_p$ for all $2 \leq p \leq n$ have vanishing products of monomial cycles in Koszul homology. The Golod condition is equivalent to the vanishing of \textit{all} products in Koszul homology and all higher-order Massey products, which we will give a brief overview of shortly. As of yet, we haven't found any examples of an ideal satisfying all $V_p$ which is not Golod.
\par
This is especially interesting considering that symmetric monomial ideals in general do not admit monomial cycle bases in Koszul homology.
\begin{example}
When $\Lambda(I) = \{(2,1,1,0)\}$, the element $[\overline{x_1x_2x_3} e_{124} - \overline{x_1 x_2 x_4} e_{123}]$ $=$ $[-\overline{x_1^2 x_2} e_{234} + \overline{x_1 x_2^2} e_{134}]$ is a non-zero element in $H_3(K^R)$ which cannot be represented by a monomial cycle.

\end{example}

\par

With the rest of this section, we will introduce some multidegree tools to prove Theorem \ref{symm principal}. A {\bf principal} symmetric monomial ideal is a symmetric monomial ideal $I = I_\lambda$ where $\Lambda(I) = \{\lambda\}$. That is,
\[
I_\lambda = (x_1^{\lambda_{\pi(1)}} \dots x_n^{\lambda_{\pi(n)}} : \pi \in \mf S_n )
\] 
Theorem \ref{partition conditions} gives a simple criterion for when $I_\lambda$ has vanishing monomial products.
\begin{proposition}
A symmetric principal monomial ideal $I_\lambda$ has vanishing products of monomial cycles in Koszul homology if and only if $\ell(\lambda) > \lfloor n/2 \rfloor$, or $\ell(\lambda) \leq \lfloor n/2 \rfloor$ and $\lambda_1 = \lambda_2$. 

\end{proposition}

\begin{proof}
If $\ell(\lambda) > \lfloor n/2 \rfloor$ then $\Lambda(I)$ satisfies $V_p^{(1)}$ vacuously for all $p$. If $\ell(\lambda) \leq \lfloor n/2 \rfloor$ and $\lambda_1 = \lambda_2$, then $\Lambda(I)$ satisfies $V_p^{(1)}$ vacuously for $\lfloor p/2 \rfloor < \ell(\lambda)$ and straightforwardly for larger $p$. 
\par
If $\ell(\lambda) \leq \lfloor n/2 \rfloor$ and $\lambda_1 > \lambda_2$, then $\Lambda(I)$ fails $V_n^{(1)}$. $V_n^{(2)}$ fails as well; it can only hold when $|\Lambda(I)| > 1$.
\end{proof}

Here we restate Theorem \ref{symm principal}.

\begin{theorem*}
A symmetric principal monomial ideal $I_\lambda$ is Golod if and only if $\ell(\lambda) > \lfloor n/2 \rfloor$ or $\ell(\lambda) \leq \lfloor n/2 \rfloor$ and $\lambda_1 = \lambda_2$.

\end{theorem*}

In order to prove this theorem, we exploit multidegrees to show that all Massey operations $\mu_r(h_1,\dots,h_r)$ vanish. For $r \geq 2$, let $h_1,\dots,h_r$ be homology classes in $H_{\geq 1}(K^R)$ with multidegrees $\mathbf a_1, \dots \mathbf a_r$ and homological degrees $m_1,\dots,m_r$. $\mu_2(h_1,h_2)$ is just the usual multiplication in Koszul homology. For $r \geq 3$, $\mu_r(h_1,\dots,h_r)$ is a set which is only defined for certain collections of homology classes, but when it is defined, each element of  $\mu_r(h_1,\dots,h_r)$ is a homology class with multidegree $\sum \mathbf a_i$ and homological degree $\sum m_i + r - 2$. These are simply properties of $\mu_r$ that we will put to use;  \cite{GulliksenTorH1969Holr} provides one of several equivalent definitions of $\mu_r(h_1,\dots,h_r)$ and a proof that $R$ is Golod if and only if the Massey operations vanish. 
\par
For monomial ideals, there is a simple criterion in terms of multidegree that guarantees vanishing of $\mu_r(h_1,\dots,h_r)$.
\begin{proposition}
\label{orthogonal}
Let $h_1,\dots,h_r \in H_{\geq 1}(K^R)$ of multidegrees $\mathbf a_1, \dots, \mathbf a_r$. If $\mu_r(h_1,\dots,h_r)$ is defined and contains not only zero, then the non-zero components of all the $\mathbf a_i$ are distinct. 

\end{proposition}

\begin{proof}
It suffices to check that this holds for the polarization of $I$. In this case, all the non-zero homology classes have square-free multidegrees. So, if $\mathbf a_i$ and $\mathbf a_j$ share a non-zero component in common for $i \neq j$, the multidegree of $\mu_r(h_1,\dots,h_r)$ is not squarefree. Hence, $\mu_r(h_1,\dots,h_r) = 0$. 

\end{proof}
{\bf Remark.} We can use Proposition \ref{orthogonal} to give an alternate proof of Proposition \ref{h1h3 vanish}: a non-monomial cycle in $H_3(K^R)$ would have a multidegree which is non-zero in each component, so all products with it must vanish.
\par
The multigraded Betti numbers of $I_\lambda$ were described in \cite{KumarAshok2013MBno}. For ease of notation, we use an alternate notation for partitions. Let $\lambda = (d_1^{p_1},\dots,d_s^{p_t},0^{p_{t+1}})$, where $d_1 > \dots > d_s > 0$, $p_1,\dots,p_t > 0$, and $d_i^{p_i}$ is shorthand for $p_i$ copies of $d_i$.

\begin{theorem}[Theorem 3.5 of \cite{KumarAshok2013MBno}]
\label{multidegs}
Let $\lambda = (d_1^{p_1},\dots,d_t^{p_t},0^{p_{t+1}})$ be a partition and let $I_\lambda$ be the symmetric principal monomial ideal with $\Lambda(I) = \{\lambda\}$. Then $\dim_k H_p(K^R)_{\mathbf a} \neq 0$ if and only if $\mathbf a$ is a permutation of a multidegree of the form
\[
(d_{i_1}^{q_1}, \dots, d_{i_j}^{q_j},0^{q_{j+1}})
\]
where $1 = i_1 < i_2 < \dots i_j \leq t$ and for all $1 \leq k \leq j$,
\[
q_1 + \dots + q_k \geq p_1 + p_2 + \dots + p_{i_k} 
\]
and 
\[
\sum_{k = 1}^j (q_1 + \dots + q_k - (p_1 + p_2 + \dots + p_{i_k})) = p-1
\]

\end{theorem}

With these tools in hand, we are ready to prove Theorem \ref{symm principal}.

\begin{proof}[Proof of Theorem \ref{symm principal}]
If $\ell(\lambda) \leq \lfloor n/2 \rfloor$ and $\lambda_1 > \lambda_2$ then there is a non-trivial product of monomial cycles by Theorem \ref{partition conditions}, so $I_\lambda$ is not Golod. 
\par
If $\ell(\lambda) > \lfloor n/2 \rfloor$ then by Proposition \ref{orthogonal} and the pigeonhole principal, all Massey operations vanish and $I_\lambda$ is Golod. So suppose $\ell(\lambda) \leq \lfloor n/2 \rfloor$ and $\lambda_1 = \lambda_2$. That is, $\lambda = (d_1^{p_1},\dots,d_t^{p_t},0^{p_{t+1}})$ where $p_1 \geq 2$ and $p_{t+1} > \lfloor n/2 \rfloor$. 
\par
Let $h_1,\dots,h_r$ be homology classes in homological degrees $m_1,\dots,m_r$ such that $\mu_r(h_1,\dots,h_r)$ contains not only zero. By Theorem \ref{multidegs}, the multidegree of an element of $\mu_r(h_1,\dots,h_r)$ is a permutation of
\[
(d_{i_1}^{q_1}, \dots, d_{i_j}^{q_j},0^{q_{j+1}})
\]
and its homological degree is  
\[
\sum_{k = 1}^j (q_1 + \dots + q_k - (p_1 + p_2 + \dots + p_{i_k})) + 1
\]
Since $\mu_r$ preserves multidegrees, each $h_s$ has as its multidegree a permutation of
\[
(d_{i_1}^{q_1^{(s)}}, \dots, d_{i_j}^{q^{(s)}_j},0^{q^{(s)}_{j+1}})
\]
where we allow $q_{i_k}^{(s)}$ to be zero for $k > 1$. Note that by Proposition \ref{orthogonal}, $\sum_s q_k^{(s)} = q_k$. 
\par 
The homological degree of $h_s$ is
\[
m_s = q_1^{(s)} - p_1 + \sum_{\substack{1 < k \leq j \\
							q_k^{(s)} > 0} } ( (q_1^{(s)} + \dots + q_k^{(s)}) - (p_1 + \dots + p_{i_k})) + 1
\]
The homological degree formula for $\mu_r$ therefore tells us that the homological degree of elements of $\mu_r(h_1,\dots,h_r)$ can also be written as $\sum m_s + r - 2$. We must, therefore, have the following equality.
\[
\sum_{s = 1}^r (q_1^{(s)} - p_1 + \sum_{\substack{1 < k \leq j \\
							q_k^{(s)} > 0} }(( q_1^{(s)} + \dots + q_k^{(s)}) - (p_1 + \dots + p_{i_k})) + 1) + r - 2 = \sum_{k = 1}^j (q_1 + \dots + q_k - (p_1 + p_2 + \dots + p_{i_k})) + 1
\]
The left hand side is equal to
\[
q_1 - rp_1 + \sum_{k = 2}^j \sum_{\substack{ 1 \leq s' \leq r \\
				q_k^{(s')} > 0}}(( q_1^{(s)} + \dots + q_k^{(s)}) - \eta_k (p_1 + \dots + p_{i_k})) + 2r - 2
\]
where $\eta_k = \#\{s' : q_k^{(s')} > 0 \} \geq 1$. The equality then reduces to
\[
(2-p_1)(r-1) - 1 =  \sum_{k = 2}^{j} ((q_1 + \dots + q_k) - \sum_{\substack{ 1 \leq s' \leq r \\
				q_k^{(s')} > 0}} ( q_1^{(s)} + \dots + q_k^{(s)}) + (\eta_k - 1)(p_1 + \dots + p_{i_k}))
\]
Since $\sum q_k^{(s)} = q_k$ and $\eta_k \geq 1$, all of the summands on the right hand side are non-negative. By assumption, $p_1 \geq 2$, so the left hand side is negative, a contradiction.
\end{proof}

\section{Ideals Admitting a Monomial Cycle Basis in Koszul Homology}
\label{section mon basis}
In this section we demonstrate that a certain class of monomial ideals with linear quotients admits a monomial basis in Koszul homology. This class includes stable ideals, matroidal ideals, and symmetric shifted ideals. Since ideals with linear quotients have a piecewise linear resolution \cite[Theorem 8.2.15]{HerzogJürgen2011Mi}, they are Golod \cite{HerzogJ.1999Clia}, and thus satisfy the inclusions of Theorem \ref{necessary cond}.

\subsection{Ideals With Linear Quotients}
We begin by recalling facts about ideals with linear quotients following \cite{HerzogTakayama}, where the notion was introduced, and the stated propositions in this subsection were proved.

For a monomial ideal $I$, let $G(I)$ denote a set of minimal monomial generators for $I$. $I$ has {\bf linear quotients} if there is some order $u_1,\dots,u_m$ of the minimal generators such that $(u_1,\dots,u_{j-1}) : u_j$ is generated by a subset of $\{x_1,\dots,x_n\}$ for all $j \geq 2$. We call $\set(u_j)$ the set $\{ i : x_i \in (u_1,\dots,u_{j-1}):u_j \}$.

\par

Suppose $I$ is an ideal with linear quotients. We denote by $M(I)$ the monomials in $I$ and let $g: M(I) \to G(I)$ be a function such that $g(u) = u_j$, where $j$ is the minimal number such that $u_j | u$. $g$ is in fact uniquely determined, and is called the {\bf decomposition function} of $I$ with respect to the given order of the monomials. We say $g$ is {\bf regular} if $\set(g(x_s u)) \subset \set(u)$ for all $u \in G(I)$ and all $s \in \set(u)$.

Minimal free resolutions of ideals with linear quotients can be constructed by `iterated mapping cones'. The free modules in such resolutions have an explicit description: 

\begin{proposition}
\label{cone basis}
Suppose $\deg u_1 \leq \dots \leq \deg u_m$. Then the minimal graded free $S$-resolution $F$ of $S/I$ has homogeneous basis
\[
F_i = \bigoplus_{\substack{u \in G(u) \\
						\sigma \subset \set(u) \\
						|\sigma| = i-1}}
				 S \cdot \gamma^u_\sigma
\]
for $i > 0$, where $\gamma^u_\sigma$ is a symbol with $\deg \gamma^u_\sigma = |\sigma| + \deg u$.
\end{proposition}

We remark that the degree-increasing condition can always be attained: Lemma 2.1 of  \cite{JAHAN2010104} says that if $I$ has linear quotients with respect to some order, then there exists an ordering $u_1,\dots,u_m$ of the generators such that $I$ has linear quotients with respect to that ordering, and that $\deg u_1 \leq \dots \leq \deg u_m$.
\par
When $I$ admits a regular decomposition function, the  maps in the minimal free resolution also have an explicit description.

\begin{proposition}
\label{delta maps}
If $I$ is an ideal with linear quotients with respect to a degree increasing order which admits a regular decomposition function, then the chain maps in the minimal free resolution $F$ of $S/I$ are given by
\[
\delta(\gamma^u_\sigma) = -\sum_{t \in \sigma} \sgn(\sigma,t) x_t \gamma^u_{\sigma \backslash t} + \sum_{t \in \sigma} \sgn(\sigma,t) \frac{x_t u}{g(x_t u)} \gamma^{g(x_t u)}_{\sigma \backslash t}
\]
for $\sigma \neq \emptyset$ and  $\delta(\gamma^u_{\emptyset}) = u$. 
\end{proposition}

\subsection{Statement and Applications of Monomial Basis Theorem}
We set up some standard notation and introduce a technical condition in order to state and apply our main theorem. For any monomial $u$ of $S$, let the support of $u$ be $\supp(u) = \{i: x_i \text{ divides } u\}$, and let $\max(u) = \max(\supp(u))$. Let $\nu_i(u)$ denote the exponent of $x_i$ in $u$. Our technical condition is the following: we say a monomial ideal $I$ with linear quotients and decomposition function $g$  {\bf has nice Koszul lifts} if for each monomial generator $u$, there exists a fixed index $\ell \in \supp(u)$ such that $x_\ell$ divides $\frac{x_t u}{g(x_t u)}$ for all $t \in \set(u)$. 

\begin{theorem}
\label{mon basis}
Let $I$ be a monomial ideal satisfying the following conditions:

\begin{enumerate}
\item $I$ has linear quotients respect to a degree increasing order.
\item $I$ admits a regular decomposition function.
\item $I$ has nice Koszul lifts.
\end{enumerate}

 Then $H_i (K^R)$ has a basis consisting of cycles of the form $[\frac{u}{x_\ell}] e_\ell \wedge e_\sigma$ where $u \in G(I)$, $\sigma \subset \set(u)$ with $|\sigma| = i-1$, and $\ell$ is a choice of index satsifying the nice Koszul lifts property for $u$.

\end{theorem}

We remark that having nice Koszul lifts is a necessary assumption. Let $J = (x_1 x_3, x_1 x_4, x_2 x_3, x_2 x_4) \subset k[x_1,x_2,x_3,x_4]$ as in Example \ref{eg}. Using this order of the generators, $J$ has linear quotients. Now $\set(x_1x_4) = \{3\}$, $\set(x_2 x_3) = \{1\}$, and $\set(x_2 x_4) = \{1,3\}$, and $\frac{x_1 (x_2 x_4)}{g(x_1 x_2 x_4)} = x_2$, $\frac{x_3 (x_2 x_4)}{g(x_3 x_2 x_4)} = x_4$. Thus, the decomposition function is regular, but $J$ does not have nice Koszul lifts. As we have seen, $[x_1 e_{234} - x_2 e_{134}]$ is a non-zero homology class which cannot be expressed as a linear combination of monomial cycles.

\par

Before proving Theorem \ref{mon basis}, we demonstrate some classes of ideals for which it applies.

\par
\underline{Matroidal ideals}: A square-free monomial ideal $I$ is {\bf matroidal} if $\{\supp(u) : u \in G(I)\}$ is the basis of a matroid. In particular, $I$ is generated in one degree. It is shown in \cite{HerzogTakayama} that $I$ has linear quotients with respect to revlex order and admits a regular decomposition function. We check that they also have nice Koszul lifts.

\begin{proposition}
\label{revlex 1 deg}
If $I$ has linear quotients with respect to revlex order, and is generated in one degree, then for any $u \in G(I)$,  $x_{\max(u)}$ divides $\frac{x_t u}{g(x_t u)}$ for all $t \in \set(u)$, where $\max(u)$ is the highest index $j$ such that $x_j$ divides $u$.
\end{proposition}

\begin{proof}
The generator $g(x_t u) \in G(I)$ is smaller than $u$ in revlex order and is of the same degree. Thus $x_{\max(u)}$ divides it.
\end{proof}

\begin{corollary}
Matroidal ideals admit a monomial cycle basis in Koszul homology.
\end{corollary}

\par
\underline{Stable Ideals}: A monomial ideal $I$ is {\bf stable} if for any $u \in G(I)$ and $i \leq \max(u)$, $x_i \frac{u}{x_{\max(u)}} \in I$. A squarefree monomial ideal $I$ is {\bf squarefree stable} if for any $u \in G(I)$ and $i < \max(u)$ with $i \not\in \supp(u)$, $x_i \frac{u}{x_{\max(u)}} \in I$. Both stable and squarefree stable ideals have linear quotients with respect to the reverse lexicographic order. In the squarefree case, $\set(u) = \{1,\dots,\max(u)-1\}$, and in the squarefree stable case $\set(u) = \{i \in \supp(u) : i < \max(u) \}$. In either case, it is easy to see that they admit regular decomposition functions (see \cite{HerzogTakayama}). \par
We apply our theorem to these ideals, giving another way of showing that stable ideals admit a monomial basis, which was originally proved in \cite{aramova} and \cite{peeva}. 

\begin{proposition}
Let $I$ be a stable or squarefree stable ideal. Then for $u \in G(I)$, $x_{\max(u)}$ divides $\frac{x_t u}{g(x_t u)}$ for all $t \in \set(u)$.
\end{proposition}

\begin{proof}
By Proposition \ref{revlex 1 deg}, it suffices to check when $\deg g(x_t u) < \deg(u)$. Let $v = g(x_t u)$ and $m = \max(u)$, and suppose toward a contradiction that $\nu_{m}(v) = \nu_m(u)$. Since $t < m$ and $\deg v < \deg(x_t u) - 1$, there is some $i < m$ such that $\nu_i(v) < \nu_i(u)$.
\par
Consider the monomial $x_i \frac{v}{x_m}$. It is greater than $v$ in revlex order since $i < m$, and it divides $x_t u$ because $\nu_i(v) < \nu_i(u)$. Thus, there is some generator $v' \in G(I)$ dividing $(x_i/x_m) v$. Now $v' < v$ and $v'$ divides $x_t u$. Thus, $v \neq g(x_t u)$, and we have a contradiction. 
\end{proof}

\begin{corollary}
Stable and squarefree stable ideals admit a monomial basis in Koszul homology.
\end{corollary}

We note that weaker versions of matroidal and stable ideals do not satisfy Theorem \ref{mon basis}. Weakly polymatroidal ideals and weakly stable ideals generated in one degree (see, e.g., \cite{MohammadiFatemeh2010Wpiw} for definitions of both) have linear quotients, but do not admit regular decomposition functions, as the ideal $(x_1x_2x_4,x_1x_2x_5,x_1x_3x_5) \subset k[x_1,\dots,x_5]$ shows. However, if it can be shown that polymatroidal ideals admit a regular decomposition function, then by Proposition \ref{revlex 1 deg}, they would satisfy the conditions of Theorem \ref{mon basis}.

\par
Not all monomial ideals which admit a monomial basis satisfy the assumptions of Theorem \ref{mon basis}. Monomial ideals in three or fewer variables (as noted above), as well as Cohen-Macaulay $p$-Borel fixed ideals \cite{p-Borel} and certain other $p$-Borel principal ideals \cite{POPESCU2008132} all have a monomial basis in Koszul homology. However, these classes of ideals do not have linear quotients in general. For example, $I = (x_1^5,x_1^4 x_2,x_1 x_2^4,x_2^5) \subset \mathbb{F}_2[x_1,x_2]$ is a Cohen-Macaulay $2$-Borel fixed principal ideal, but it does not have a linear resolution. Hence, $I$ does not have linear quotients.

\par
\underline{Symmetric Shifted Ideals:}
A special class of of symmetric monomial ideals are the {\bf symmetric shifted ideals}, introduced in \cite{BIERMANN2020312}. They are defined by the property that if $\lambda \in \Lambda(I)$, and $\lambda_k < \lambda_1$ for $k > 1$, then $x^\lambda(\frac{x_k}{x_1}) \in I$ (where $x^\lambda = x_1^{\lambda_1} \dots x_n^{\lambda_n}$). 
\par
Theorem 3.2 in \cite{BIERMANN2020312} says that symmetric shifted ideals have linear quotients given a particular monomial order, which we will define. Let $<_{\lex}$ be the graded lexicographic order on $\Z^n$. That is, $\bf a <_{\lex} \bf b$ if $|\bf a| < |\bf b|$ or $|\bf a| = |\bf b|$ and the leftmost non-zero entry of $\bf a - \bf b$ is negative. For two monomials $u$ and $v$, we can write them uniquely as $u = \pi(x^\lambda)$ and $v = \rho(x^\mu)$, where $\lambda,\mu$ are partitions and $\pi,\rho$ are permutations. The monomial order we will use is $u \prec v$ if $\lambda <_{\lex} \mu$ or if $\lambda = \mu$ and $u <_{\lex} v$. The motto: grade by partitions first, then by lex order. Ordering the generators of a symmetric shifted ideal $I$ in increasing $\prec$-order, $I$ has linear quotients. Furthermore, we have a nice description of $\set(u)$ for $u \in G(I)$. Let $u = \pi(x^\lambda)$ and $m(u) := \max\{i : \nu_i(u) = \lambda_1\}$. Then
\[
\set(u) = \{i : \nu_i(u) < \lambda_1 - 1\} \cup \{j : \nu_j(u) = \lambda_1 - 1 \text{ and } j < m(u) \}
\]
\par
It's clear that $\prec$ is a degree increasing order. The next two propositions show that symmetric ideals satisfy the rest of the criteria for Theorem \ref{mon basis}.

\begin{proposition}
Let $I$ be a symmetric shifted ideal. Then $I$ admits a regular decomposition function.

\end{proposition}

\begin{proof}
Let $u \in G(I)$, $s \in \set(u)$, and $v = g(x_s u)$.  Let $\lambda,\mu$ partitions such that $u = \pi(x^\lambda)$ and $v = \rho(x^\mu)$. Let $x_s u = v c$ where $c$ is some monomial. By Lemma 1.7(a) of \cite{HerzogTakayama}, $\set(v) \cap \supp(c) = \emptyset$. So, $x_k$ divides $c$ only if $\nu_k(v) = \mu_1$ or $\nu_k(v) = \mu_{1}-1$ and $k > m(v)$. 
\par
Note that the highest exponent it $x_s u = vc$ is $\lambda_1$ since $s \in \set(u)$. Since $v$ divides $x_s u$, $\mu_1 \leq \lambda_1$. Let $t \in \set(v)$ such that $t \neq s$. If $\nu_t(v) < \mu_1 - 1$, then $\nu_t(c) = 0$, so $\nu_t(u) = \nu_t(v) < \lambda_1 - 1$, and $t \in \set(u)$. Assume $\nu_t(v) = \mu_1 - 1$ with  $t < m(v)$, and furthermore assume that $\mu_1 = \lambda_1$ (else $\nu_t(v) = \nu_t(u) < \lambda_1-1$). In this case, $c$ is squarefree, and consists only of variables $x_k$ such that $\nu_k(v) = \mu_1 - 1$ and $k > m(v)$. So, $m(v c) > m(v)$. But $m(vc) = m(x_s u) = m(u)$. So, $t < m(u)$ and thus $t \in \set(u)$. 

\end{proof}

\begin{proposition}
Let $I$ be a symmetric shifted ideal with $u \in G(I)$. Then $x_{m(u)}$ divides $\frac{x_s u}{g(x_s u)}$ for all $s \in \set(u)$.

\end{proposition}

\begin{proof}
Let $u = \pi(x^\lambda)$ and let $v = g(x_s u)$. Suppose toward a contradiction that $\nu_{m(u)}(v) = \lambda_1$. Our goal is to show that, for some $t \in \set(v)$, $\frac{x_t v}{x_{m(v)}}$ divides $x_s u$. Now $\frac{x_t v}{x_{m(v)}} \in I$ because $I$ is symmetric shifted, and furthermore $\frac{x_t v}{x_{m(v)}} \prec v$. This implies that $v$ is not in fact $g(x_s u)$ and we have our contradiction.
\par
Since $v$ divides $x_s u$, there exists a $t$ such that $\nu_t(v) < \nu_t(x_s u)$. Clearly $\frac{x_t v}{x_{m(v)}}$ also divides $x_s u$. We wish to show that in fact $t$ is in $\set(v)$. Since $s \in \set(u)$, $\nu_s(x_s u) \leq \lambda_1$ and if equality holds, $s < m(u)$. This, plus the assumption that $\nu_{m(u)}(v) = \lambda_1$, implies that $m(v) = m(u)$. 
\par
Now if $t > m(u) = m(v)$, $\nu_t(v) \leq \lambda_1 - 1$. But equality can't hold, else $\nu_t(x_s u) = \lambda_1$, and either $m(u) > m(v)$ or $\nu_s(x_s u) = \lambda_1-1$ and $s > m(u)$, a contradiction regardless. So, $\nu_t(v) < \lambda_1 - 1$ and therefore $t \in \set(v)$. 
\par
Now suppose $t < m(v)$. If $t = s$, then $v$ not only divides $x_s u$ but also $u$, so $u \not\in G(I)$. So  assume that $t \neq s$. Then $\nu_t(v) < \nu_t(x_s u) = \nu_t(u) \leq \lambda_1$. So $t < m(v) =  m(u)$ and $\nu_t(v) < \lambda_1$. Thus, $t \in \set(v)$. 
\par
So, $t \in \set(v)$ and $\frac{x_t v}{x_{m(v)}}$ divides $x_s u$, which is our desired contradiction.

\end{proof}

\subsection{Proof of Theorem \ref{mon basis}}
\label{proof mon basis}

The proof technique is a standard way of demonstrating a basis for Koszul homology, relying on the `balancing of Tor'. Let $(F,\delta)$ be the minimal $S$-resolution of $R = S/I$. There are graded quasi-isomorphisms 
\[
K^R = R \otimes K \xleftarrow{\sim} F \otimes K \xrightarrow{\sim} F \otimes k
\]
inducing an isomorphism on homologies $H_i(F \otimes k) \to H_i (K^R)$, which is computed by subsequent lifting along $\partial$ and mapping down by $\delta$ along the portion of the double complex $F \otimes K$ given in the following diagram:

\begin{tikzcd}
                                                            &                                                             &                                                         & F_{i} \otimes K_0 \arrow[r, dashed, "\partial"] \arrow[d, "\delta"] & F_{i} \otimes k \arrow[r, dashed] & H_i(F \otimes k) \\
                                                            &                                                             & F_{i-1} \otimes K_1 \arrow[r, dashed, "\partial"] \arrow[d, "\delta"] & F_{i-1} \otimes K_0                                 &                   \\
                                                            & F_{i-2} \otimes K_2 \arrow[r, dashed, "\partial"] \arrow[d, "\delta"] & F_{i-2} \otimes K_1                         &                                               &                   \\
F_0 \otimes K_i \arrow[r, dashed, "\partial"] \arrow[d, "\delta"] & \cdots                                                      &                                                         &                                               &                   \\
R \otimes K_i \arrow[d]                               &                                                             &                                                         &                                               &                  
\\
H_i(K^R)
\end{tikzcd}

Under this isomorphism, a multigraded basis of $H_i(F \otimes k)$ will get mapped to a multigraded basis of $H_i(K^R)$. Under the assumptions of Theorem \ref{mon basis}, we have an explicit basis of $F_i \otimes k$ and a description of $\delta$ by Proposition \ref{cone basis} and Proposition \ref{delta maps}. So, our method of proof is simply to check that $\gamma^u_\sigma \otimes 1$ can be mapped to $[\frac{u}{x_\ell}] e_\ell \wedge e_{\sigma}$ under this isomorphism. 

\begin{proof}[Proof of Theorem \ref{mon basis}]
 The basis element $\gamma^u_\sigma \otimes 1 \in H_i(F \otimes k)$ can be lifted to $\gamma^u_\sigma \otimes 1 \in F \otimes S$.  Mapping down by $\delta$, we get
\[
\delta(\gamma^u_\sigma \otimes 1) = -\sum_{t \in \sigma} \sgn(\sigma,t) x_t \gamma^u_{\sigma \backslash t} + \sum_{t \in \sigma} \sgn(\sigma,t) \frac{x_t u}{g(x_t u)} \gamma^{g(x_t u)}_{\sigma \backslash t}
\]
By the assumption that $I$ has nice Koszul lifts, the element
\[
-\sum_{t \in \sigma} \sgn(\sigma,t) \gamma^u_{\sigma \backslash t}e_t + \sum_{t \in \sigma} \sgn(\sigma,t) \frac{x_t u}{g(x_t u) x_\ell} \gamma^{g(x_t u)}_{\sigma \backslash t} e_\ell \in F_{i-1} \otimes K_1
\] 
is a lift of $\delta(\gamma^u_\sigma \otimes 1)$ (this is where ``nice Koszul lifts" gets its name). From here we proceed by induction. For $a \geq 1$, let $\Gamma_a \in F_{i - a} \otimes K_a$ be an element in $(\partial^{-1} \delta)^a(\gamma^u_\sigma \otimes 1)$. We will prove by induction that we can choose $\Gamma_a$ such that
\[
\Gamma_a = \sum_{\substack{\tau \subset \sigma \\
				|\tau| = a}} (-1)^a \epsilon_\tau \gamma^u_{\sigma \backslash t} e_{\tau}
\]
\[
+ \sum_{\substack{ \tau' \subset \sigma \\
					|\tau'| = a - 1}} \sum_{t' \in \sigma \backslash \tau'} (-1)^{a + 1} \epsilon_{\tau'} \sgn (\sigma \backslash \tau',t') \frac{x_{t'}{u}}{g(x_{t'}u)x_\ell} \gamma^{g(x_{t'} u)}_{\sigma \backslash \{t,t'\}} e_\ell \wedge e_{\tau'} 
\]
where $\epsilon_\tau := \prod_{t \in \tau} \sgn(\sigma,t)$ for a set $\tau \subset \sigma$. The base case was just done. For the induction step, we compute $\delta(\Gamma_a)$:
\[
\sum_{\substack{\tau \subset \sigma \\
				|\tau| = a}} \sum_{t \in \sigma \backslash \tau} (-1)^{a + 1} \epsilon_\tau \sgn(\sigma \backslash \tau , t) [ x_t \gamma^u_{\sigma \backslash \tau \backslash t} - \frac{x_t u}{g(x_t u)} \gamma^{g(x_t u)}_{\sigma \backslash \tau \backslash t}] e_{\tau}
\]
\[
+ \sum_{\substack{ \tau' \subset \sigma \\
					|\tau'| = a - 1}} \sum_{t' \in \sigma \backslash \tau'} \sum_{s' \in \sigma \backslash \tau' \backslash t'} (-1)^{a} \epsilon_{\tau'} \sgn(\sigma \backslash \tau',t') \sgn(\sigma \backslash \tau' \backslash t',s') [\frac{x_{s'} x_{t'}{u}}{g(x_{t'}u)x_\ell} \gamma^{g(x_{t'} u)}_{\sigma \backslash \tau' \backslash \{t',s'\}}
\]
\[
 - \frac{x_{t'}{u}}{g(x_{t'}u)x_\ell} \frac{x_{s'}g(x_{t'}u)}{g(x_{s'}g(x_{t'} u))} \gamma^{g(x_{s'} g(x_{t'} u))}_{\sigma \backslash \tau' \backslash \{t',s'\}}] e_\ell \wedge e_{\tau'}
\]
We label four types of terms appearing in this sum:
\[
\tag{Type 1}
(-1)^{a + 1} \epsilon_\tau \sgn(\sigma \backslash \tau , t)  x_t \gamma^u_{\sigma \backslash \tau \backslash t} e_\tau
\]
\[
\tag{Type 2}
(-1)^{a} \epsilon_\tau \sgn(\sigma \backslash \tau , t)\frac{x_t u}{g(x_t u)} \gamma^{g(x_t u)}_{\sigma \backslash \tau \backslash t} e_{\tau}
\]
\[
\tag{Type 3}
(-1)^{a} \epsilon_{\tau'} \sgn(\sigma \backslash \tau',t') \sgn(\sigma \backslash \tau' \backslash t',s') \frac{x_{s'} x_{t'}{u}}{g(x_{t'}u)x_\ell} \gamma^{g(x_{t'} u)}_{\sigma \backslash \tau' \backslash \{t',s'\}} e_\ell \wedge e_{\tau'}
\]
\[
\tag{Type 4}
(-1)^{a+1} \epsilon_{\tau'} \sgn(\sigma \backslash \tau',t') \sgn(\sigma \backslash \tau' \backslash t',s') \frac{x_{t'}{u}}{g(x_{t'}u)x_\ell} \frac{x_{s'}g(x_{t'}u)}{g(x_{s'}g(x_{t'} u))} \gamma^{g(x_{s'} g(x_{t'} u))}_{\sigma \backslash \tau' \backslash \{t',s'\}}] e_\ell \wedge e_{\tau'}
\]

We claim that terms of Type 4 cancel with each other.  Swapping $s'$ and $t'$ switches the sign, and by Lemma 1.11 of \cite{HerzogTakayama} says that $g(x_s g(x_t u)) = g(x_t g(x_s u))$ for $s,t \in \set(u)$. So indeed terms of Type 4 cancel.
\par
Now, we just need to show that for our given choice of $\Gamma_{a+1}$, $\partial(\Gamma_{a+1})$ is the appropriate sum of elements of types 1,2, and 3. We have
\[
\Gamma_{a+1} = \sum_{\substack{\rho \subset \sigma \\
				|\rho| = a+1}} (-1)^{a + 1} \epsilon_\rho \gamma^u_{\sigma \backslash \rho} e_{\rho}
+ \sum_{\substack{ \rho' \subset \sigma \\
					|\rho'| = a}} \sum_{b \in \sigma \backslash \rho'} (-1)^{a} \epsilon_{\rho'} \sgn(\sigma \backslash \rho',b) \frac{x_{b}{u}}{g(x_{b}u)x_\ell} \gamma^{g(x_b u)}_{\sigma \backslash \rho \backslash b} e_\ell \wedge e_{\rho'} 
\]
The image under $\partial$ of the first sum is
\[
\sum_{\substack{\rho \subset \sigma \\
				|\rho| = a+1}}
				\sum_{c \in \rho} (-1)^{a + 1} \epsilon_\rho \sgn(\rho,c) x_c \gamma^u_{\sigma \backslash c} e_{\rho \backslash c}
\]
which will consist of all terms of Type 1. To see this, we let $t = c$ and $\tau = \rho \backslash c$. We show that we get the appropriate signs. That is, we need $(-1)^{a + 1} \epsilon_{\tau \cup t} \sgn(\tau \cup t,t)$ and $(-1)^{a + 1} \epsilon_{\tau}  \sgn(\sigma \backslash \tau,t)$ to be equal. This is equivalent to $\sgn(\sigma,t) \sgn(\tau \cup t,t) = \sgn(\sigma \backslash \tau,t)$. Now 
\[
 \#\{m \in \sigma \backslash \tau : m < t\} = \#\{m \in \sigma : m < t\} - \#\{m \in \tau  \cup t: m < t\}
 \] so by definition of $\sgn$, the previous statement is true.
\par 
Next, we wish to show that the image of
\[
\sum_{\substack{ \rho' \subset \sigma \\
					|\rho'| = a}} \sum_{b \in \sigma \backslash \rho'} (-1)^{a}  \epsilon_{\rho'} \sgn(\sigma \backslash \rho',b) \frac{x_{b}{u}}{g(x_{b}u)x_m} \gamma^{g(x_b u)}_{\sigma \backslash \rho' \backslash b} e_\ell \wedge e_{\rho'} 
\]
under $\partial$, which is 
\[
\sum_{\substack{ \rho' \subset \sigma \\
					|\rho'| = a}} \sum_{b \in \sigma \backslash \rho'} (-1)^{a}  \epsilon_{\rho'}  \sgn(\sigma \backslash \rho',b) \frac{x_{b}{u}}{g(x_{b}u)} \gamma^{g(x_b u)}_{\sigma \backslash \rho' \backslash b} e_{\rho'} 
\]
\[
+ \sum_{\substack{ \rho' \subset \sigma \\
					|\rho'| = a}} \sum_{b' \in \sigma \backslash \rho'} \sum_{c' \in \rho'} (-1)^{a+1} \epsilon_{\rho'}  \sgn(\sigma \backslash \rho',b')  \sgn(\rho',c') \frac{x_{c'}x_{b'}{u}}{g(x_{b'}u)x_\ell} \gamma^{g(x_{b'} u)}_{ \sigma \backslash \rho' \backslash \{b',c'\}} e_\ell \wedge e_{\rho' \backslash c'} 
\]
consists of all terms of Types 2 and 3. The first sum, if we let $\tau = \rho'$ and $t = b$, are simply the terms of Type 2. In the triple sum, if we let $\tau' = \rho' \backslash c'$, $t' = b'$, and $s' = c'$, we get the necessary terms of Type 3, assuming the signs are correct. To see that they are, we need to show that 
\[
(-1)^a \epsilon_{\tau'} \sgn(\sigma \backslash \tau',t') \sgn(\sigma \backslash \tau' \backslash t',s') = (-1)^{a+1} \epsilon_{\tau' \cup c'} \sgn(\sigma \backslash \tau' \backslash c',b') \sgn(\tau \cup c', c')
\]
that is,
\[
\sgn(\sigma \backslash \rho' \cup c',b') \sgn(\sigma \backslash \rho' \backslash b' \cup c',c') = -\sgn(\sigma,c') \sgn(\rho', c') \sgn(\sigma \backslash \rho',b') 
\]
If $b' < c'$, this is equivalent to
\[
\sgn(\sigma \backslash \rho',b') (-\sgn(\sigma \backslash \rho' \cup c',c') = -\sgn(\sigma,c')\sgn(\rho',c')\sgn(\sigma \backslash \rho',b')
\]
which is true by the same consideration of the definition of $\sgn$ as above. The argument is similar when $b' > c'$. Either way, the sign is correct.
\par
The induction step is now done. By induction, $\Gamma_i$ is 
\[
(-1)^{i-1} \epsilon_\sigma \gamma^u_{\emptyset} e_{\sigma}
+ \sum_{t \in \sigma} (-1)^{i}  \epsilon_{\sigma \backslash t} \frac{x_{t}{u}}{g(x_{t}u)x_\ell} \gamma^{g(x_t u)}_{\emptyset} e_\ell \wedge e_{\sigma \backslash t} 
\]
whose image under $\delta$ is 
\[
(-1)^{i-1} \epsilon_\sigma u e_{\sigma}
+ \sum_{t \in \sigma} (-1)^{i}  \epsilon_{\sigma} \sgn(\sigma,t) \frac{x_{t}{u}}{x_\ell}e_\ell \wedge e_{\sigma \backslash t} = \partial((-1)^{i-1} \epsilon_\sigma \frac{u}{x_\ell} e_\ell \wedge e_{\sigma})
\]
Now $(-1)^{i-1}\epsilon_\sigma \frac{u}{x_\ell} e_\ell \wedge e_{\sigma}$ descends to the desired basis element of $H_{i}(K^R)$, and the theorem is proved.

\end{proof}

\bibliography{mybibliography}

\end{document}